\documentclass[final,5p,times,number]{elsarticle}
\usepackage{graphicx,amsthm}
\newtheorem{theorem}{Theorem}

\newtheorem{example}{Example}
\usepackage[latin1]{inputenc}
\usepackage{color}
\usepackage{array,algorithm,dsfont,colortbl,units,multirow}
\usepackage{algorithmic}
\usepackage[usenames,dvipsnames]{xcolor}
\usepackage{amsmath,amssymb,hyperref,bm}
\usepackage{float}
\usepackage[caption=false]{subfig}
\def\hyph{-\penalty0\hskip0pt\relax}
\def\T{{\mathsf T}}
\def\y{{\boldsymbol y}}

\def\h{{\boldsymbol h}}
\def\m{{\boldsymbol m}}
\def\v{{\boldsymbol v}}
\def\w{{\boldsymbol w}}

\def\x{{\boldsymbol x}}
\def\g{{\boldsymbol g}}
\def\q{{\boldsymbol q}}
\def\z{{\boldsymbol z}}
\def\H{{\boldsymbol H}}

\def\bSigma{\boldsymbol{\Sigma}}
\def\E{\mathbb{E}}
\def\Tr{{\rm Tr}}
\def\bpsi{{\boldsymbol \psi}}
\def\bphi{{\boldsymbol \phi}}
\def\B{\mathcal{B}}
\def\F{\mathcal{F}}
\def\R{\mathcal{R}}

\definecolor{Gray}{gray}{0.75}
\definecolor{green}{rgb}{0,0.5,0}
\newcolumntype{g}{>{\columncolor{Gray}}c}
\newcounter{assumption}
\newenvironment{assumption}{%
\refstepcounter{assumption}%
\edef\@currentlabelname{{\bf{A}}\arabic{assumption}}%
\par\noindent{\bf{(Assumption \arabic{assumption})}}\itshape\
}{%
\hfill\qed\par\noindent%
\ignorespacesafterend
}


\newcounter{parentnumber}

\journal{Neurocomputing}

\begin{document}
\begin{frontmatter}
\title{On Distributed Online Classification in the Midst of Concept Drifts}

\author[ucla]{Zaid~J.~Towfic}
\ead{ztowfic@ee.ucla.edu}
\author[ucla]{Jianshu~Chen}
\ead{jshchen@ee.ucla.edu}
\author[ucla]{Ali~H.~Sayed\corref{cor1}}
\ead{sayed@ee.ucla.edu}

\address[ucla]{Electrical Engineering Department\\
University of California\\
Los Angeles, CA 90095, USA}
\cortext[cor1]{Corresponding author}
%\editor{}

%***********************************************************************
% END OF AUTHORS INFORMATION AREA
%***********************************************************************

\begin{abstract}%
In this work, we analyze the generalization ability of distributed online learning algorithms under stationary and non-stationary environments. We derive bounds for the excess-risk attained by each node in a connected network of learners and study the performance advantage that diffusion strategies have over individual non-cooperative processing. We conduct extensive simulations to illustrate the results.
\end{abstract}

\end{frontmatter}
\allowdisplaybreaks
\section{Introduction}
Stochastic gradient algorithms provide powerful and iterative techniques for the solution of optimization problems \cite{Polyak}. In many situations of interest, the objective function is in the form of the expectation of a convex loss function over the distribution of the input data. Such situations arise in machine learning applications, where the input data are features to a classifier and their associated class labels. For example, the goal of a binary classifier is to predict the label ($\pm 1$) given a vector of features that describes an observation (or, equivalently, to separate two classes based on their feature vector descriptions). The classifier achieves this goal by learning a classification rule based on a cost function that penalizes incorrect classification according to some criterion. The cost function is usually referred to as the {\em{risk}} \cite[p. 20]{Vapnik}, and it measures the generalization error that is achieved by the classifier (that is, it measures how well a classifier is able to predict the labels associated with feature vectors that have not yet been observed). The \emph
{excess-risk} is defined as the difference between the risk achieved by the classifier given its classification rule and the smallest risk achievable by the classifier over all possible classification rules. It is  critical to study the excess-risk performance of a classifier in order to understand how the classifier will perform on future data compared to the best possible classifier.

Several works in the literature study excess-risk {\em indirectly} by deriving regret bounds and then relating these bounds to excess-risk \cite{OCO,Kakade}. This two-step procedure suffers from two drawbacks: 1) the procedure is targeted at algorithms that utilize diminishing step-sizes, which are not useful for non{\hyph}stationary environments, and 2) the second step that relates the regret to excess-risk is not tight, and it has been shown that online learning algorithms that utilize diminishing step-sizes can achieve better performance than dictated by the indirect analysis \cite{hazan2}. In this article, we study the excess-risk directly and for {\em constant} step-sizes in order to cope with non-stationary environments. Among other results, we establish that a constant step-size {\em distributed} algorithm of the diffusion type can achieve arbitrarily small excess-risk for appropriately chosen step-sizes in stationary environments.

Distributed stochastic learning seeks to leverage cooperation between nodes over a network in order to optimize the overall network risk without the need for coordination or supervision from a central entity that has access to the entire data (like features and labels) from all nodes. Such distributed sche{\hyph}mes are particularly useful when the data sampled by the nodes cannot be shared broadly due to privacy or communication constraints. In \cite{zinkevich_parallelized}, an algorithm is developed that requires a central node or server to poll the optimization estimates from all the nodes at the end of a time horizon. This approach is not fully distributed and is not able to track changes in the generating distribution without restarting the algorithm. One fully distributed learning algorithm appears in \cite{yan} where the global cost is chosen as the aggregate regret over the network of learners. The scheme of \cite{yan} consists of a single consensus-type iteration of the form \eqref{eq:consensus} further ahead and is similar to the schemes proposed in \cite{nedic2} for distributed optimization; the analysis in \cite{nedic2} is limited to the noise-free case. In the estimation literature, references \cite{cassio,federico,book_chapter} proposed distributed schemes that rely on diffusion rather than consensus iterations. Diffusion strategies allow for information to diffuse more readily through the network, and they enhance stability, convergence, and robustness in comparison to consensus strategies \cite{shine_diffusion_consensus}. Diffusion strategies consist of two steps: a combination step that averages the estimates in the local neighborhood of an agent, and an adaptation step that incorporates new information into the local estimator of each agent. The net result is that information is diffused across the network adaptively and in real-time. The diffusion approach was generalized in \cite{Jianshu_common_wo} for general strongly convex cost functions and constant step-sizes. 

In comparison to the earlier work on diffusion adaptation \cite{cassio,federico,Xiaochuan}, we study in this work  the excess-risk performance of general strongly-convex risk functions as opposed to mean-square error performance. This level of generality allows us to study the excess-risk performance for regularized logistic regression in addition to the delta rule (square loss). In comparison to \cite{Jianshu_common_wo,MLSP2}, we study the \emph{tracking} performance of the distributed classifiers when the optimizer is time-varying. The effective tracking of a drifting concept is only possible when the algorithm utilizes a constant step-size as opposed to diminishing step-sizes as used in \cite{yan,MLSP2}. Even for stationary environments, we show that the proposed algorithms obtain better performance than the non-cooperative solution for any strongly-convex risk function when some mild assumptions hold regarding the noise process.

One of the main objectives of this work is therefore to study the generalization ability of diffusion strategies  when the distribution from which the data arises is {\em{time-varying}}. When the statistics of the input to the classifier change, the classifier must adjust its classification rule in order to accurately classify the data arising from the new distribution --- see Fig.~\ref{fig:rotating_hyperplane} further ahead. In the context of machine learning, this change in the best possible classification rule for the non-stationary data is referred to as \emph{concept drift} \cite{WidmerKubat,angP2P,Jenn2010}. We desire to answer one key question: is it possible to obtain convergence results for the distributed learning algorithms to show tracking of a changing optimal classification rule? We discover that the answer is in the affirmative under some assumptions. 

\section{Problem Formulation and Algorithm}
It is assumed that a classifier receives samples $\x_i$ over time $i$ arising from some underlying statistical distribution. A loss function $Q(w,\x_i)$ is associated with $\x_i$ and depends on an $M\times 1$ parameter vector $w$. The classifier wishes to minimize the risk function over $w$, which is defined as the expected loss \cite[p. 20]{Vapnik}:
\begin{align}
	J(w) = \mathbb{E}_{{\x_i}}\{Q(w,{\x_i})\}\quad\quad\textrm{(risk function)}
	\label{eq:J_ind}
\end{align}
It is usually assumed that the data $\x_i$ are independent and identically distributed (i.i.d.).  It is also assumed that the risk function $J(w)$ is strongly convex. Obviously, when the data are stationary, then the risk $J(w)$ will not depend on time. Observe that we are denoting random quantities by using the \textbf{boldface} notation, which will be our convention in this article. One example that fits into this formulation is logistic regression \cite[p. 117]{Theodoridis} where the cost function is defined as:
\begin{align}
	J(w) \triangleq  \E_{\{\y_{i},\h_{i}\}} \left\{\frac{\rho}{2} \|w\|^2 + \log(1+e^{-\y_{i} \h_{i}^\T w})\right\}
	\label{eq:log-loss}
\end{align}
where $\x_{i}$ in \eqref{eq:J_ind} is now defined as the aggregate data $\{\y_{i},\h_{i}\}$ where $\y_{i}$ denotes the scalar label for feature vector $\h_{i} \in \mathbb{R}^M$. Moreover, $\rho$ is a positive scalar regularization parameter. We will utilize logistic regression in the simulation section to illustrate  our analysis.

In order to assess and compare the performance of algorithms that are used to minimize \eqref{eq:J_ind}, we adopt the excess-risk (ER) measure, which is defined as follows:
\begin{align}
	\mathrm{ER}(i) \triangleq \E\{J(\w_{i-1}) - J(w^o)\}\quad\quad\textrm{(excess-risk)}
	\label{eq:ER_ind}
\end{align}
where $w^o$ is the optimizer of \eqref{eq:J_ind} over all $w$ in the feasible space: \begin{align}
	w^o &\triangleq \arg\min_w\ J(w)
\end{align}
and $\w_{i-1}$ is the estimator for $w^o$ \emph{available} at time $i-1$. The reason why the excess-risk is evaluated using $\w_{i-1}$ is that excess-risk measures the generalization ability of a classifier on future data \emph{before} observing the data. The estimate $\w_{i}$, as we will see in Alg.~\ref{alg:diffusion}, would incorporate data from time $i$. The variable $\w_{i-1}$ is generally a random quantity since it will be influenced by randomness in the data arising from the gradient vector approximations that are used during the development of stochastic gradient procedures; the gradient approximations are referred to as instantaneous approximations in the adaptive literature \cite{haykin,sayed}, and are also sometimes called the gradient oracle in the machine learning literature (see, e.g., \cite{hazan2,Agrawal_Information_Theoretic}).  The expectation in \eqref{eq:ER_ind} is taken over the distribution of $\w_{i-1}$.

Considerable research has focused on deriving bounds for the excess-risk in gradient descent procedures for stand-alone classifiers. In this work, we pursue two extensions to these results. First, we assume that we have a network of $N$ learners connected by means of some topology. The only requirement is that the network be connected, meaning that there is a path connecting any two arbitrary agents in the network; this path may be through a sequence of other agents. Figure \ref{fig:sample_network} illustrates one such network. The nodes in the shaded region represent the neighborhood of node $1$ (denoted by $\mathcal{N}_1$). Second, we allow the statistical distribution of the data $\x_i$ to change with time. This change causes the optimizer $w^o$ to drift.
\begin{figure}
	\centering
	\includegraphics[width=0.25\textwidth]{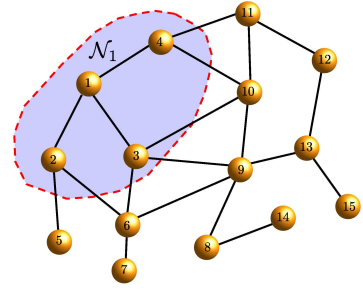}
	\caption{A connected network. The shaded region represents the neighborhood of node $1$.}
	\label{fig:sample_network}
\end{figure}

We associate with each agent $k$ in the network an individual loss function $Q_k(w,\x_{k,i})$ evaluated at the corresponding feature vector, $\x_{k,i}$. The corresponding strongly-convex risk function is generally time-varying and given by:
\begin{align}
J_{k,i}(w) \triangleq \E_{\x_{k,i}} \{Q_k(w,\x_{k,i})\}
\label{eq:J_k_def}
\end{align}
We further consider a global network risk, which is defined as the average of the individual risks over all nodes:
\begin{align}
	J^\mathrm{glob}_i(w) \triangleq \frac{1}{N}\sum_{k=1}^N J_{k,i}(w)\quad\quad\textrm{(network risk function)}
	\label{eq:j_glob}
\end{align}
The excess-risk at node $k$ is defined as:
\begin{align}
	\mathrm{ER}_k(i) \triangleq  \!\E\!\left\{J_{k,i}(\w_{k,i-1}) - J_{k,i}(w^o_i)\right\} \ \textrm{(excess-risk at node $k$)}
	\label{eq:ER_k}
\end{align}
while the network excess-risk is defined as:
\begin{align}
	\mathrm{ER}(i) &\triangleq  \frac{1}{N} \sum_{k=1}^N \mathrm{ER}_k(i) \nonumber\\
	&=\!\E\!\left\{\frac{1}{N}\!\!\sum_{k=1}^N \!J_{k,i}(\w_{k,i-1}) - J^\mathrm{glob}_i(w^o_i)\right\} \ \textrm{(network excess-risk)}
	\label{eq:ER_dist}
\end{align}
where in both cases 
\begin{align}
	w^o_i &\triangleq \underset{w}{\arg\min}\ J^\mathrm{glob}_i(w)
	\label{eq:w_o_glob}
\end{align}
When the distribution is stationary and $N=1$ (i.e., a network with a single node), we see that our formulation collapses to the one described by \eqref{eq:J_ind} and \eqref{eq:ER_ind}. We assume that the optimizer $w^o_i$ in \eqref{eq:w_o_glob} is also the optimizer of the component risk functions $J_{k,i}(w)$, i.e.,
\begin{equation}
\boxed{
	w^o_i = \underset{w}{\arg\min}\ J^\mathrm{glob}_i(w) = \underset{w}{\arg\min}\ J_{k,i}(w), \;\;\;\;k=1,2,\ldots,N
}
\label{eq:common_minimizer}
\end{equation}
This condition is satisfied when the nodes are sampling data arising from a time-varying distribution defined by the same set of parameters. That is, when the data do not reflect local preferences, then \eqref{eq:common_minimizer} is usually satisfied. When the environment is stationary and $w^o$ is therefore constant, reference \cite{Jianshu_common_wo} derived the distributed algorithm listed in the table below for the solution of \eqref{eq:w_o_glob}. One of the objectives of this work is to show that this same algorithm can also be used to track drifting concepts $w^o_i$. We will evaluate how well it performs in this case. 

\begin{algorithm}[!ht]
\caption{\!{\bf{:}}\ Diffusion strategy for risk optimization}
\label{alg:diffusion}
{\small{
\begin{algorithmic}
	\STATE Consider the problem of optimizing the network risk function \eqref{eq:j_glob} in a distributed manner. For each node $k$, let ${\cal N}_k$ denote the set of its neighbors, namely, all nodes with which node $k$ can share information (including node $k$ itself). Select non-negative coefficients $\{a_{1,\ell k}\}$, $\{c_{\ell k}\}$, and $\{a_{2,\ell k}\}$ that satisfy
	\begin{align}
	\begin{cases}
		\underset{\ell \in \mathcal{N}_k}{\sum} a_{1,\ell k} = \underset{\ell \in \mathcal{N}_k}{\sum} c_{k \ell} = \underset{\ell \in \mathcal{N}_k}{\sum} a_{2,\ell k} = 1\\
		a_{1,\ell k} = c_{\ell k} = a_{2,\ell k} = 0, \textrm{when}\ \ell\notin {\cal N}_k
	\end{cases}
	\label{eq:left_stochastic}
	\end{align}
	Each node $k$ starts with an initial weight estimate $w_{k,0}$ and repeats over $i\geq 1$:
	\begin{align}
	\phi_{k,i-1} &= \sum_{\ell=1}^N a_{1,\ell k} w_{\ell,i-1} \label{eq:C1}\\
	\psi_{k,i} &= \phi_{k,i-1} - \mu \sum_{\ell = 1}^N c_{\ell k} \widehat{\nabla} J_{\ell,i-1}(\phi_{k,i-1}) \label{eq:A}\\
	w_{k,i}   &= \sum_{\ell=1}^N a_{2,\ell k} \psi_{\ell,i} \label{eq:C2}
	\end{align}
	where $\widehat{\nabla} J_{\ell,i-1} (\cdot)$ is an approximation for the true gradient vector $\nabla J_{\ell,i-1}(\cdot)$, and $\mu$ is a positive step-size parameter.
\end{algorithmic}}}
\end{algorithm}
In Alg.~\ref{alg:diffusion}, each node $k$ interacts with its one-hop neighbors and updates its parameter estimate using approximations for the true gradient vector. The coefficients $a_{1,\ell k}$, $c_{\ell k}$, and $a_{2,\ell k}$ are non-negative scalars corresponding to the $(\ell,k)$ entries of $N\times N$ matrices $A_1$, $C$, and $A_2$, respectively. In view of the requirement \eqref{eq:left_stochastic}, the matrices $A_1$ and $A_2$ are left-stochastic while the matrix $C$ is right-stochastic. Different choices for $\{A_1,A_2,C\}$ lead to different variations of the algorithm. For example, setting $A_1=I$ and $A_2=A$ leads to an Adapt-then-Combine (ATC) strategy where the first step is an adaptation step, followed by combination:
\begin{align}
	\begin{cases}
		\psi_{k,i} &= w_{k,i-1} - \mu\  \displaystyle{\sum_{\ell = 1}^N} c_{\ell k} \widehat{\nabla} J_{\ell,i-1}(w_{k,i-1})\\
		w_{k,i}   &= \displaystyle{\sum_{\ell = 1}^N} a_{\ell k} \psi_{\ell,i}
	\end{cases}\quad\quad\textrm{(ATC)}
	\label{eq:ATC1}
\end{align}
On the other hand, setting $A_1=A$ and $A_2=I$ leads to a Combine-then-Adapt (CTA) strategy where adaptation follows combination:
\begin{align}
 	\begin{cases}
 		\phi_{k,i-1} &= \displaystyle{\sum_{\ell = 1}^N} a_{\ell k} w_{\ell,i-1}\\
	w_{k,i} &= \phi_{k,i-1} - \mu\ \displaystyle{\sum_{\ell = 1}^N} c_{\ell k} \widehat{\nabla} J_{\ell,i-1}(\phi_{k,i-1})\\
 	\end{cases}\quad\quad\textrm{(CTA)}
\end{align}
In either the ATC or CTA versions, we can further set $C=I_N$. In this case, the adaptation step would rely only on the gradient vector at node $k$, e.g.,
\begin{align}
	&\begin{cases}
		\psi_{k,i} &= w_{k,i-1} - \mu \widehat{\nabla} J_{k,i-1}(w_{k,i-1})\\
		w_{k,i}   &= \displaystyle{\sum_{\ell = 1}^N} a_{\ell k} \psi_{\ell,i}
	\end{cases}\quad\quad\textrm{(ATC)}\label{eq:ATC2}\\
	&\begin{cases}
 		\phi_{k,i-1} &= \displaystyle{\sum_{\ell = 1}^N} a_{\ell k} w_{\ell,i-1}\\
	w_{k,i} &= \phi_{k,i-1} - \mu \widehat{\nabla} J_{k,i-1}(\phi_{k,i-1})\\
 	\end{cases}\quad\quad\!\!\!\!\textrm{(CTA)}
 	\label{eq:CTA_C=I}
\end{align}
Likewise, setting $A_1\!=\!A_2\!=\!C\!=\!I_N$ leads to the non-cooperative mode of operation where each node optimizes its risk individually and independently of the other nodes:
\begin{align}
w_{k,i} &= w_{k,i-1} - \mu \widehat{\nabla} J_{k,i-1}(w_{k,i-1})\quad\quad\textrm{(no cooperation)}
\label{eq:no_coop}
\end{align}
It is important to note that diffusion strategies are different in a fundamental way from the algorithm presented in \cite{yan,nedic2}, which has the form:
\begin{align}
	w_{k,i} = \sum_{\ell \in \mathcal{N}_k} a_{\ell k} w_{\ell,i-1} - \mu \widehat{\nabla} J_{k,i-1}(w_{k,i-1})
	\label{eq:consensus}
\end{align}
For instance, comparing with \eqref{eq:CTA_C=I}, we see that one critical difference is that the gradient vector used in \eqref{eq:consensus} is evaluated at $w_{k,i-1}$, whereas it is evaluated at $\phi_{k,i-1}$ in \eqref{eq:CTA_C=I}. In this way, information beyond the immediate neighborhood of node $k$ influences the updates at $k$ more effectively in the diffusion case \eqref{eq:CTA_C=I}. This order of the computations has an important implication on the dynamics of the resulting algorithm. For example, it can be verified that even if all individual learners are stable in the mean-square sense, a network of learners using an update of the form \eqref{eq:consensus} can become unstable, while the same network using the diffusion updates \eqref{eq:ATC2}-\eqref{eq:CTA_C=I} will always be stable regardless of the choice of the matrix $A$---see \cite{shine_diffusion_consensus}. In the next section, we establish a relationship between excess-risk and mean-square-error (MSE) and provide the main assumptions for the rest of the manuscript.

\section{Excess-risk, Weighted MSE, and Main Assumptions}
\label{sec:assumptions}
\noindent Introduce the prediction and filtering weighted mean-square-errors (MSEs):
\begin{align}
\E\|\tilde{\w}_{k,i}^p\|^2_T &\triangleq \E\|w^o_i-\w_{k,i-1}\|^2_T \label{eq:WMSE_pred}\\
\E\|\tilde{\w}_{k,i}^f\|^2_T &\triangleq \E\|w^o_i-\w_{k,i}\|^2_T \label{eq:WMSE_filt}
\end{align}
where $\|x\|^2_T \triangleq x^\T T x$ for any positive semi-definite weighting matrix $T$. When the environment is stationary (i.e., when $w^o_i = w^o$ for all $i$), we notice that there is effectively no difference in the filtering and prediction MSE in steady-state. The reason why we need to introduce the two errors is that the excess-risk \eqref{eq:ER_dist} requires that a previous estimate (prior to observing the current data) is used to evaluate the performance of the classifier. We will see shortly that under the non-stationary model we adopt in this work, the prediction and filtering MSEs are related to each other. To proceed, we introduce the following assumption regarding the Hessian matrices of the functions $J_{k,i}(w)$.\\
\begin{assumption}%[\bf{A\ref{ass:HessianAssumption}}]
	\label{ass:HessianAssumption}
	The Hessian matrices of the individual risk functions $J_{k,i}(w)$ are uniformly bounded from below and from above for all $k \in \{1,\ldots,N\}$ and time $i$:
	\begin{equation}
		\lambda_{\min} I_M \leq \nabla^2 J_{k,i}(w) \leq \lambda_{\max} {I_M}
		\label{eq:HessianAssumption}
	\end{equation}
	where $0 < \lambda_{\min} \leq \lambda_{\max} < \infty$.  %\qed
\end{assumption}
Assumption \ref{ass:HessianAssumption} essentially states that the risk functions encountered for all times and at all nodes can be upper and lower-bounded by a quadratic cost. The lower-bound on the Hessians in \eqref{eq:HessianAssumption} translates into saying that the functions $J_{k,i}(w)$ are strongly-convex \cite[pp. 9-10]{Polyak}. For example, the risk function \eqref{eq:log-loss} for regularized logistic regression satisfies Assumption \ref{ass:HessianAssumption}.

We may note that in \cite{hazan2,microsoft,Nedic}, the risk functions are assumed to have bounded gradient vectors (as opposed to bounded Hessian matrices). Clearly, there are cost functions (such as quadratic cost functions) where the gradient is not bounded while the Hessian is; for this reason, our Assumption \ref{ass:HessianAssumption} enables the subsequent analysis to be applicable to a larger class of risk functions.

Now consider the excess-risk suffered at iteration $i$ at node $k$. It can be expressed as:
\begin{align}
	\textrm{ER}_k(i) &= \mathbb{E}\{J_{k,i}(\w_{k,i-1}) - J_{k,i}(w^o_i)\} \nonumber\\
	& \stackrel{\mathrm{(a)}}{=} \mathbb{E}\left\{-\int_0^1 \nabla J_{k,i}(w^o_i- t \ \tilde{\w}_{k,i}^p)^\T dt \  \tilde{\w}_{k,i}^p \right\} \nonumber\\
	&						 \stackrel{\mathrm{(b)}}{=}  \mathbb{E}\left\{\!-\!\int_0^1 \nabla J_{k,i}(w^o_i)^\T dt \ \tilde{\w}_{k,i}^p + \right. \nonumber\\
	&\left.\!\quad\quad\quad\quad \tilde{\w}_{k,i}^{p \T} \!\!\left[\int_0^1 \!\!t \!\int_0^1 \!\!\!\nabla^2 J_{k,i}(w^o_i \!-\! s \ t \  \tilde{\w}_{k,i}^p)ds dt\right] \!\tilde{\w}_{k,i}^p\right\} \nonumber\\
	& \stackrel{\mathrm{(c)}}{=} \mathbb{E}\left\{\tilde{\w}_{k,i}^{p \T} \left[\int_0^1\!\! t\!\! \int_0^1\!\!\! \nabla^2 J_{k,i}(w^o_i \!-\! s \ t \ \tilde{\w}_{k,i})ds dt \right] \tilde{\w}_{k,i}^p\right\} \nonumber\\
	& \triangleq \mathbb{E}\{\|\tilde{\w}_{k,i}^p\|_{{\bm{T}}_{k,i}}^2\} \label{eq:MSE_ind}\\
	&\stackrel{\mathrm{(d)}}{\leq} \frac{\lambda_{\max}}{2} \mathbb{E}\| \tilde{\w}_{k,i}^p\|^2
	\label{eq:function_diff_bound}
\end{align}
where
\begin{align}
	\tilde{\w}_{k,i}^p \triangleq w^o_i - \w_{k,i-1}
\end{align}
Steps (a) and (b) in the sequence of calculations that led to \eqref{eq:function_diff_bound} are a consequence of the following mean-value theorem from \cite[p. 24]{Polyak}:
\begin{equation}
	f(a+b) = f(a) + \int_0^1 \nabla f(a + t\cdot b)^\T\ dt \cdot b
	\label{eq:polyak1}
\end{equation}
Step (c) is a consequence of the fact that $w^o_i$ optimizes $J_{k,i}(w)$ so that $\nabla\!J_{k,i}(w^o_i) \!\!=\! 0$. Step (d) is due to \eqref{eq:HessianAssumption}, where we defined the weighting matrix as:
\begin{align}
	\bm{T}_{k,i} \triangleq \left[\int_0^1 t \int_0^1 \nabla^2 J_{k,i}(w^o_i - s \ t \ \tilde{\w}_{k,i}^p)ds dt \right]
	\label{eq:weight_T}
\end{align}
It follows from \eqref{eq:function_diff_bound} that if the MSE at all nodes is uniformly bounded over time, then the network excess-risk \eqref{eq:ER_dist} will be bounded by the same bound scaled by $\lambda_{\max}/2$. For this reason, it is justified that we examine the mean-square-error performance of the diffusion strategy \eqref{eq:C1}-\eqref{eq:C2} under stationary and non-stationary conditions, and then use these results to bound the network excess-risk by using the relation:
\begin{align}
	\mathrm{ER}(i) &= \frac{1}{N} \sum_{k=1}^N \E\left\{\|\tilde{\w}_{k,i}^p\|^2_{\bm{T}_{k,i}}\right\} = \E\left\{\|\tilde{\w}_{i}^p\|^2_{\bm{\mathcal{T}}_i}\right\}
	\label{eq:ER_cal_T}
\end{align}
where
\begin{align}
\tilde{\w}_i^p &\triangleq \mbox{\rm col}\{\tilde{\w}_{1,i}^p,\tilde{\w}_{2,i}^p,\ldots,\tilde{\w}_{N,i}^p\}
\end{align}
collects the prediction weight error vectors across all nodes, and
\begin{align}
	\bm{\mathcal{T}}_i \triangleq \frac{1}{N} \mathrm{diag}\{\bm{T}_{1,i},\ldots,\bm{T}_{N,i}\}
	\label{eq:def_T_weight_matrix}
\end{align}
A key point to stress here is that the network excess-risk {\em{is}} the weighted network MSE when the weighting matrix is set to the above $\bm{\mathcal{T}}_i$. In order to perform the mean-square-error analysis of the network, we need to introduce some assumptions. First, we introduce a modeling assumption regarding the perturbed gradient vectors used by the algorithm.\\
\begin{assumption}%[A\ref{ass:noiseModeling}]
	\label{ass:noiseModeling}
	We model the perturbed gradient vector as:
	\begin{align}
		\widehat{\nabla}_w J(\w) = \nabla_w J(\w) + \v_{k,i}(\w)
		\label{eq:grad_model}
	\end{align}
	where, conditioned on the past history of the estimators $\{\w_{k,j}\}$ for $j \leq i-1$ and all $k$, the gradient noise $\v_{k,i}(\w)$ satisfies:
\begin{align}
		&\mathbb{E}\{\v_{k,i}(\w) | \mathcal{H}_{i-1}\} = 0\\
		&\mathbb{E}\{\|\v_{k,i}(\w)\|^2\} \leq \alpha \cdot\mathbb{E}\|w^o_i-\w\|^2 + \sigma_v^2 \label{eq:noise_assumption_variance}
	\end{align}	
	for some $\alpha\geq 0$, $\sigma_v^2 \geq 0$, and where $\mathcal{H}_{i-1} \triangleq \{\w_{k,j}:k = 1 , \ldots, N\ \mathrm{and}\ j \leq i-1\}$.
\end{assumption}
Assumption \ref{ass:noiseModeling} models the perturbed gradient vector as the true gradient plus some noise. This noise consists of two parts: relative noise and absolute noise. The variance of the relative noise component depends on the distance between the estimate and the optimum at time $i$ ($w^o_i$). On the other hand, the variance of the absolute noise term is represented by the factor $\sigma_v^2$ in \eqref{eq:noise_assumption_variance}. As the quality of the weight estimate by the node improves, the power of the relative noise component decreases. The second part of the noise bound in \eqref{eq:noise_assumption_variance} refers to absolute noise; this component does not depend on the current weight estimate and is bounded by $\sigma_v^2$. The absolute noise guarantees that there will always remain some perturbation on the estimated gradient vectors even when the gradient is evaluated at the optimum $w^o_i$. We may remark that, in contrast to Assumption 2, most earlier references  \cite{hazan2,microsoft,Nedic} in the literature assumed only the presence of the absolute noise term and ignored relative noise. The following example is from \cite{Jianshu_common_wo}.

\begin{example}
{\rm Consider \texttt{ADALINE} \cite[p. 103]{Theodoridis}, \cite{LMS}. Let the binary class label at node $k$ and time $i$ be denoted by $\y_{k,i} \in \{-1,+1\}$. Let the feature vector at node $k$ and time $i$ be denoted by $\h_{k,i} \in \mathbb{R}^M$. \texttt{ADALINE} optimizes the quadratic loss:
\begin{align}
	Q_k(w,\y_{k,i},\h_{k,i}) \triangleq \left|\y_{k,i}-\h_{k,i}^\T w\right|^2
	\label{eq:Loss_Adaline}
\end{align}
The risk function is then the expectation of the loss in \eqref{eq:Loss_Adaline}:
\begin{align}
	J_k(w) \triangleq \mathbb{E} \left|\y_{k,i}-\h_{k,i}^\T w\right|^2
	\label{eq:J_adaline}
\end{align}
Let the data satisfy the linear model:
\begin{align}
	\y_{k,i} \triangleq \h_{k,i}^\T w + \z_{k}(i)
	\label{eq:linear_model}
\end{align}
where the feature vectors $\{\h_{k,i}\}$ are assumed to be zero-mean with a constant covariance matrix $R_{h,k} \triangleq \E \{\h_{k,i} \h_{k,i}^\T\}$. The noise sequence $\{\z_k(i)\}$ is assumed to be zero-mean and white with constant variance $\sigma_{z,k}^2$. The optimal solution $w^o$ that minimizes \eqref{eq:J_adaline} satisfies the normal equations:
\begin{align}
	r_{h y,k} = R_{h,k} w^o
	\label{eq:normal_equations}
\end{align}
where $r_{h y,k} \triangleq \E \{\h_{k,i} \y_{k,i}\}$. The feature vectors and noise are assumed to be independent over nodes and time. One instantaneous approximation for the gradient vector is:
\begin{align}
	\widehat{\nabla} J_k(w) &= -2\h_{k,i} (\y_{k,i} - \h_{k,i}^\T w)
	\label{eq:ADALINE_inst_approx}
\end{align}
Using \eqref{eq:linear_model}-\eqref{eq:ADALINE_inst_approx} and \eqref{eq:grad_model}, we have that the gradient noise satisfies:
\begin{align}
	\v_{k,i}(\w) &= \widehat{\nabla} J_k(w) - \nabla J_k(w) \nonumber\\
				&= 2 (R_{h,k} - \h_{k,i} \h_{k,i}^\T) (w^o - \w) - 2 \h_{k,i} \z_{k}(i)
\end{align}
We then have that:
\begin{align}
	\E \{\v_{k,i} | \mathcal{H}_{i-1}\} &= 0\\
	\E \|\v_{k,i}(\w)\|^2 &\!\leq\! 4 \E \left\{\left(\sigma_{\max}\left(R_{h,k} -  \h_{k,i} \h_{k,i}^\T\right)\right)^2\right\} \cdot \E \| w^o - \w \|^2 + \nonumber\\
	&\quad\! 4 \Tr(R_{h,k}) \sigma_{z}^2 \label{eq:adaline_var}
\end{align}
for all $\w \in \mathcal{H}_{i-1}$ where $\sigma_{\max}(A)$ denotes the maximum singular value of its matrix argument $A$. Therefore the \texttt{ADALINE} algorithm satisfies Assumption \ref{ass:noiseModeling} under \eqref{eq:linear_model}. Note that both noise terms (relative and absolute) appear on the right hand side of \eqref{eq:adaline_var}.\hfill\qed
}
\end{example}

We shall distinguish between two scenarios in our analysis. In the first case, we assume the optimizer $w^o$ does not change with time (i.e., we assume stationarity). In the second case, we assume the optimizer $w^o_i$ varies slowly with time according to a random walk model.  \\
\begin{assumption}%[A\ref{ass:stationary_distribution}]
	\label{ass:stationary_distribution}
	The data process $\x_i$ is stationary. This implies that the risk functions $J_{k,i}(w)$ defined in \eqref{eq:J_k_def} are time-invariant:
	\begin{align}
		J_{k,i}(w) = J_k(w),\quad \textrm{for all\ } i
	\end{align}	
	In addition, this implies that the optimizer $w^o_i$ of the network risk function and all individual risk functions is constant, $w^o_i = w^o$ for all $i$.
%	\qed
\end{assumption}
When the environment is non-stationary, we shall assume instead a random-walk model for the minimizer $w^o_i$.\\
\begin{assumption}%[A\ref{ass:random_walk}]
\label{ass:random_walk}
In the non-stationary case, the time-varying optimal vector $\w^o_i$ is modeled as a random walk:
\begin{align}
\w^o_i \triangleq \w^o_{i-1} + \q_i
\label{eq:random_walk_model}
\end{align}
where the zero-mean sequence $\q_i$ has covariance $\E \{\q_i \q_i^\T\} = Q$ and is independent of the quantities $\{\v_{k}(\w_{k,j}), \q_{j}\}$ for all $j<i$. The mean of $\w^o_i$ is set to $\E\{\w^o_i\} = w^o$.
%\qed
\end{assumption}

Observe that the time-varying optimizer $\w^o_i$ is now denoted by a boldface letter due to the addition of the random noise component $\q_i$; furthermore, the expectation in the definition of excess-risk \eqref{eq:ER_dist} will now operate over this randomness as well. In machine learning, this random-walk model was used in \cite{Jenn2010} to describe the concept drift of a classifier with a moving hyperplane. This model is also commonly used to evaluate the tracking performance of adaptive filters \cite[pp. 271-272]{sayed}.

Assumption \ref{ass:random_walk} models the desired set of parameters $\w^o_i$ as a non-stationary first-order autoregressive (AR(1)) process. Such AR(1) processes are commonly used to model non-stationary behavior in various contexts such as adaptive filtering \cite{sayed} and financial data modeling \cite[pp.~142--146]{AnalysisofEconomicData}, \cite[pp.~72--73]{TimeSeriesAnalysis}. Similar models have been used in other contexts such as web searching. For example, the original PageRank algorithm, used by the Google search engine, uses a naive ``random surfer'' that models an average user that traverses a random walk over the graph of Internet webpages \cite{PageRank}. Although the model is simplistic in terms of modeling the shifts of a user's interest, it has been demonstrated to achieve excellent page sorting capability.

Given Assumption \ref{ass:random_walk}, we can relate the prediction and filtering errors introduced in \eqref{eq:WMSE_pred}-\eqref{eq:WMSE_filt}:
\begin{align}
	\E \|\tilde{\w}_{k,i}^p\|_T^2 &= \E\|\w_i^o - \w_{k,i-1}\|_T^2 \nonumber\\
	&= \E\|\w_{i-1}^o - \w_{k,i-1} + \q_i\|_T^2 \nonumber \\
	&= \E\|\w_{i-1}^o - \w_{k,i-1} \|_T^2 + \E\|q_i\|_T^2 \nonumber\\
	&= \E \|\tilde{\w}_{k,i}^f\|_T^2 + \Tr(Q T)
	\label{eq:relationship_prediction_filtering}
\end{align}
This means that in order to show that the prediction error $\E \|\tilde{\w}_{k,i}^p\|_T^2$ remains bounded for the diffusion algorithm, it is sufficient to analyze the filtering error $\E \|\tilde{\w}_{k,i}^f\|_T^2$ and show that it is bounded.

We will further introduce an assumption to be used later in the article to derive relationships between the performance of the algorithms described by \eqref{eq:ATC1}-\eqref{eq:no_coop}.
\begin{assumption}%[A\ref{ass:common_cost}]
	\label{ass:common_cost}
	The risk functions across the nodes are identical:
	\begin{align}
		J_{k,i}(w) = J_{i}(w), \quad k \in \{1,\ldots,N\}
	\end{align} %\qed
\end{assumption}
Assumption \ref{ass:common_cost} states that all nodes have the same risk function, but this does not mean that the nodes will receive the same {\em{data realizations}}. Assumption \ref{ass:common_cost} is satisfied when the nodes utilize the same loss function $Q(\cdot,\cdot)$ and receive data arising independently from the same distribution.

\section{Stationary Environments}
In this section, we focus on obtaining convergence results for the excess-risk for the distributed diffusion strategy \eqref{eq:C1}-\eqref{eq:C2} under stationary conditions. First, we show that the diffusion algorithm can achieve arbitrarily small excess-risk given appropriately chosen step-sizes.
\begin{theorem}[Excess-risk for stationary environments is $O(\mu)$]
\label{thm:convergence}
\hfill\\Let Assumptions \ref{ass:HessianAssumption}-\ref{ass:stationary_distribution} hold. Given a small constant step-size $\mu$ that satisfies:
	\begin{align}
		0 < \mu < \min\left\{\frac{2 \lambda_{\max}}{\lambda_{\max}^2 + \alpha}, \frac{2 \lambda_{\min}}{\lambda_{\min}^2 + \alpha}\right\}
		\label{eq:mu_condition}
	\end{align}	
	Then, Algorithm \ref{alg:diffusion} achieves arbitrarily small excess-risk at each node $k$, i.e.:
	\begin{equation}
		\underset{i \rightarrow \infty}{\limsup}\ \mathrm{ER}_k(i) \leq \epsilon
		\label{eq:RegretBound}
	\end{equation}
where $\epsilon$ is defined as:
\begin{equation}
	\epsilon \triangleq \frac{\sigma_v^2}{4}\cdot\frac{\lambda_{\max}}{\lambda_{\min}} \cdot \mu
	\label{eq:epsilon}
\end{equation}
and is directly proportional to the step-size $\mu$. Since each node can achieve an arbitrarily small excess-risk, the network excess-risk in \eqref{eq:ER_dist} can also be made arbitrarily small.
\end{theorem}
\begin{proof}
Given Assumption \ref{ass:stationary_distribution}, we have that the risk functions $J_{k,i}(w)$ are time-invariant ($J_{k,i}(w) = J_k(w)$) and the optimizer is constant ($w^o_i = w^o$) for all time $i$. Furthermore, we have from \eqref{eq:function_diff_bound} that the excess-risk at node $k$ is bounded by the scaled mean-square-error:
\begin{align}
    \mathbb{E}_{\w} \{J_{k}(\w_{k,i-1}) - J_{k}(w^o)\} &\leq \frac{\lambda_{\max}}{2} \mathbb{E}_{\w} \|\tilde{\w}_{k,i-1}\|^2
                              \label{eq:function_diff_bound2}
\end{align}
We now appeal to results from \cite{Jianshu_common_wo} (Theorem 1, Equations (67) and (72)) where it is shown that for $\mu$ satisfying \eqref{eq:mu_condition} it holds that
\begin{align}
	 \underset{i\rightarrow \infty}{\limsup}\  \mathbb{E}_\w\| \tilde{\w}_{k,i-1} \|^2 \leq \frac{\sigma_v^2}{2 \lambda_{\min}} \mu, \quad k \in \{1,\ldots,N\}
	 \label{eq:regret_term_bound}
\end{align}
Therefore, if we define $\epsilon$ as in \eqref{eq:epsilon} and further bound \eqref{eq:function_diff_bound2} using \eqref{eq:regret_term_bound}, we obtain \eqref{eq:RegretBound}.
\end{proof}
Result \eqref{eq:function_diff_bound2} implies that when the environment is stationary, meaning the optimizer $w^o_i$ is actually fixed for all time, then the excess-risk attained at each node in the network will be bounded by an arbitrarily small quantity that is proportional to the step-size $\mu$ when \eqref{eq:mu_condition} is satisfied. As we will see in the next section, this arbitrary reduction of the MSE is not generally possible for non-stationary environments.

In addition to Theorem \ref{thm:convergence}, it is possible to approximate the excess-risk at node $k$ (and also the network excess-risk) at steady state for sufficiently small $\mu$.
\begin{theorem}[Steady-state approximation for excess-risk]
\label{thm:SS_stationary}
\hfill\\Let Assumptions \ref{ass:HessianAssumption}-\ref{ass:stationary_distribution} hold. For small step-sizes that satisfy \eqref{eq:mu_condition}, the steady-state network excess-risk from \eqref{eq:ER_cal_T} for Alg.~\ref{alg:diffusion} can be approximated by:
	\begin{align}
	\lim_{i\rightarrow\infty} \mathrm{ER}(i) &\approx \mu^2 \mathrm{vec}\left(\mathcal{A}_2^\T \R_v^\T \mathcal{A}_2\right)^\T(I-\F)^{-1} \mathrm{vec}(\mathcal{T})
\label{eq:ER_SS_approx}
	\end{align}
	where 
	\begin{align}
		\F &\triangleq {\mathcal B}^\T \otimes {\mathcal B}^T    \label{eq:def_F}\\
		{\mathcal B}&={\mathcal A}_2^\T(I_{MN}-\mu {\mathcal D}){\mathcal A}_1^\T\\
		\mathcal{A}_1 &\triangleq A_1 \otimes I_M\\
		\mathcal{A}_2 &\triangleq A_2 \otimes I_M\\
		\mathcal{D} &\triangleq \sum_{\ell=1}^N \mathrm{diag}\left\{c_{\ell,1},\ldots,c_{\ell,N}\right\} \otimes \nabla^2 J_\ell(w^o) 	\label{eq:D_infty1}\\
		\mathcal{T} &= \frac{1}{2N} \mathrm{diag}\left\{\nabla^2 J_1(w^o), \ldots, \nabla^2 J_N(w^o)\right\}
	\end{align}
	and the symbol $\otimes$ denotes the Kronecker product operation \cite[p.~139]{laub} and $\textrm{vec}(\cdot)$ refers to the operation that stacks the columns of its matrix argument on top of each other \cite[p.~145]{laub}. Furthermore, the matrix $\R_v$ in \eqref{eq:ER_SS_approx} is defined as the covariance matrix of the vector $\g_i$:
	\begin{align}
		\g_i \triangleq \sum_{\ell=1}^N \mathrm{col}\{c_{\ell 1}\v_{\ell,i}(w^o),\dots,c_{\ell N}\v_{\ell,i}(w^o)\}
		\label{eq:bg_i}
	\end{align}
	That is, $\R_v \triangleq \E\{\g_i\g_i^T\}$.
\end{theorem}
\begin{proof}
From \eqref{eq:ER_cal_T}, we notice that the excess-risk can be evaluated as the weighted mean-square-error with weight matrix $\bm{\mathcal{T}}_i$ defined in \eqref{eq:def_T_weight_matrix} and \eqref{eq:weight_T}. When the environment is stationary and $w_i^o = w^o$ is constant, the weight matrix $\bm{T}_{k,i}$ in \eqref{eq:weight_T} becomes:
\begin{align}
    \bm{T}_{k,i} \triangleq \left[\int_0^1 t \int_0^1 \nabla^2 J_{k}(w^o - s \ t \ \tilde{\w}_{k,i-1})ds dt \right]
\end{align}
Furthermore, due to Theorem \ref{thm:convergence}, we have that the mean-square value of $\tilde{\w}_{k,i-1}$ is small for small step-size $\mu$ and large $i$. This implies that we can approximate the weight matrix $\bm{T}_{k,i}$ by
\begin{align}
     \bm{T}_{k,i} &\approx T_{k} = \left[\int_0^1 t \int_0^1 \nabla^2 J_{k,i}(w^o)ds\,dt \right] = \frac{1}{2} \nabla^2 J_{k}(w^o) \quad\!\! (\textrm{small } \mu)
     \label{eq:T_k_infty}
\end{align}
for large $i$ and small $\mu$. In other words, the matrix $\bm{T}_{k,i}$ becomes approximately deterministic and is given by $T_k$ at steady-state. Therefore, the matrix $\bm{\mathcal{T}}_i$ defined in \eqref{eq:def_T_weight_matrix} can, in steady-state, be approximated by the deterministic matrix:
\begin{align}
    \bm{\mathcal{T}}_i \approx \mathcal{T} = \frac{1}{N} \mathrm{diag}\left\{T_{1},\ldots,T_{N}\right\}
    \label{eq:T_infty}
\end{align}
We can now utilize results from \cite{Jianshu_common_wo} to approximate the excess-risk at steady-state. Using (103) from \cite{Jianshu_common_wo} we can write: 
\begin{align}
    \lim_{i\rightarrow\infty} \mathbb{E} \|\tilde{\w}_i\|^2_{\Sigma - \mathcal{B}^\T \Sigma \mathcal{B}} \approx \Tr(\mu^2 {\mathcal A}_2^\T {\mathcal R}_v^\T {\mathcal A}_2 \Sigma)
    \label{eq:Jianshu_approx}
\end{align}
where $\Sigma$ is an arbitrary positive semi-definite matrix that we are free to choose. Assume we choose $\Sigma$ such that 
\begin{align}
\Sigma - {\cal B}^\T \Sigma {\cal B}={\cal T}
\end{align}
for some ${\mathcal T}$, which could be equal to \eqref{eq:T_infty} or some other choice (see Table \ref{tbl:metrics}). If we stack the columns of $\Sigma$ into a vector $\sigma={\mathrm{vec}}(\Sigma)$, then the above equality implies that $\sigma$ is chosen as 
\begin{align}
\sigma=({\mathcal I}-{\mathcal F})^{-1}{\mathrm{vec}}({\mathcal T})
\end{align}
The matrix $(I-{\mathcal F})$ is invertible for sufficiently small step-sizes (see App. C in \cite{Jianshu_common_wo}). Therefore, we conclude from \eqref{eq:Jianshu_approx} that
\begin{align}
\lim_{i\rightarrow\infty} \E\|\tilde{\w}_{i}\|^{2}_\mathcal{T} &\approx \mu^2 \mathrm{vec}\left(\mathcal{A}_2^\T \R_v^\T \mathcal{A}_2\right)^\T(I-\F)^{-1} \mathrm{vec}(\mathcal{T})
\label{eq:ER_approx_2}
\end{align}
Different choices for ${\mathcal T}$ are possible in \eqref{eq:ER_approx_2}. For example, if we select ${\mathcal T}$ as in \eqref{eq:T_infty}, then \eqref{eq:ER_approx_2} would approximate the network excess-risk \eqref{eq:ER_dist} at steady-state. Table \ref{tbl:metrics} lists other choices for ${\mathcal T}$. 
\end{proof}
Different metrics can be evaluated by choosing $\mathcal{T}$ appropriately. For instance, in order to evaluate the mean-square-error at node $k$, we let $\mathcal{T} = E_{kk}$ where $E_{kk}$ is the zero matrix with a single $1$ in the $k$-th diagonal element. On the other hand, in order to evaluate the excess-risk at node $k$, we let $\mathcal{T} = E_{kk} \otimes T_{k}$ where $T_{k} = \frac{1}{2} \nabla^2 J_k(w^o)$.

\begin{table*}
\caption{Choice of $\mathcal{T}$ for the evaluation of different performance
metrics. $E_{kk}$ indicates the all zero matrix with a single $1$
in the $k$-th diagonal element.}
\begin{centering}
\begin{tabular}{|c|c|c|c|c|}
\hline
Metric & $\mathbb{E}_{\w}\left\{ J_{k}(\w_{k,\infty})-J_{k}(w^{o})\right\} $ & ${\displaystyle \frac{1}{N}\sum_{k=1}^{N}\mathbb{E}_{\w}\left\{ J_{k}(\w_{k,\infty})-J_{k}(w^{o})\right\} }$ & $\mathbb{E}_{\w}\left\{ \|\tilde{\w}_{k,\infty}\|^{2}\right\} $ & ${\displaystyle \frac{1}{N}\sum_{k=1}^{N}\mathbb{E}_{\w}\left\{ \|\tilde{\w}_{k,\infty}\|^{2}\right\} }$\tabularnewline
\hline
\hline
$\mathcal{T}$ & $E_{kk}\otimes T_{k}$ & $\frac{1}{N}\mathrm{diag}\{T_{1},\ldots,T_{N}\}$ & $E_{kk}$ & $\frac{1}{N}I_{MN}$\tabularnewline
\hline
\end{tabular}
\par\end{centering}
\label{tbl:metrics}
\end{table*}

It is possible to compare the performance of Alg.~\ref{alg:diffusion} against that of non-cooperative processing \eqref{eq:no_coop} when the nodes act individually and do not cooperate with each other. The non-cooperative case \eqref{eq:no_coop} is a special case of Alg.~\ref{alg:diffusion} when the matrices $\{A_1,A_2,C\}$ are all set equal to the identity matrix.
\begin{theorem}[Cooperation versus no-cooperation]
\label{thm:ATC_CTA_Ind}
	\hfill\\Let Assumptions \ref{ass:HessianAssumption}-\ref{ass:stationary_distribution} hold. In addition, let Assumption \ref{ass:common_cost} hold so that all nodes have the same risk function. Assume the step-size satisfies condition \eqref{eq:mu_condition}. Consider the ATC, CTA, and the non-cooperative algorithms \eqref{eq:ATC2}-\eqref{eq:no_coop} with $C=I$. Assume the combination matrix $A$ in the ATC and CTA cases is chosen to be doubly-stochastic, meaning that $A^\T \mathds{1}=\mathds{1}$ and $A \mathds{1}=\mathds{1}$. When Assumption \ref{ass:common_cost} holds, the weighting matrix $\mathcal{T}$ in \eqref{eq:T_infty} has the form $\mathcal{T} = \frac{1}{2N} I_N \otimes \nabla^2 J(w^o)$. Under these conditions, the steady-state network excess-risk satisfies:
	\begin{align}
		\mathrm{ER}_{\mathrm{ATC}} \leq \mathrm{ER}_{\mathrm{CTA}} \leq \mathrm{ER}_{\mathrm{ind}}
	\end{align}
	where $\mathrm{ER}_{\mathrm{ATC}}$ is the steady-state excess-risk when the $\mathrm{ATC}$ algorithm is executed, $\mathrm{ER}_{\mathrm{CTA}}$ is the steady-state excess-risk when the $\mathrm{CTA}$ algorithm is executed, and $\mathrm{ER}_{\mathrm{ind}}$ steady-state excess-risk when the nodes do not cooperate with each other.
\end{theorem}
\begin{proof}
See \ref{app:app4}.
\end{proof}
From Theorem \ref{thm:ATC_CTA_Ind}, we observe that the Adapt-then-Combine (ATC) algorithm outperforms the Combine-then-Adapt (CTA) strategy, which in turn outperforms the non-cooperative strategy for any doubly-stochastic combination matrix $A$. The reason ATC outperforms CTA is because adaptation precedes combination in ATC so that improved weight estimates are aggregated in the combination step. Nevertheless, as the step-size $\mu$ becomes smaller, then the gap between the ATC and CTA algorithms also becomes smaller (see Fig.~\ref{fig:ER_webspam_large_mu}-\ref{fig:ER_webspam_small_mu} further ahead).

In the next section, we study the performance of the diffusion strategy \eqref{eq:C1}-\eqref{eq:C2} when the optimizer $w^o_i$ is changing according to Assumption \ref{ass:random_walk}. We will establish that the excess-risk is bounded even under this scenario.

\section{Non-Stationary Environments}
\label{sec:Non-Stationary}
In the previous section, we showed that if we use a constant step-size, the mean-square-error and network excess-risk  for Alg.~\ref{alg:diffusion} can be made arbitrarily small by choosing the step-size to be sufficiently small. However, reduction of the excess-risk is not always possible in non-stationary environments. In order to arrive at meaningful bounds for the tracking performance of the algorithm, we will utilize the random-walk model from Assumption \ref{ass:random_walk}.
\begin{theorem}[Asymptotic ER bound for non-stationary data]
\label{thm:tracking}
	\hfill\\Let Assumptions \ref{ass:HessianAssumption}-\ref{ass:noiseModeling} and \ref{ass:random_walk}  hold, and choose a constant step-size that satisfies, as $i\rightarrow\infty$:
	\begin{align}
		0 < \mu < \frac{2 \lambda_{\min} C_*}{\|C\|_1^2(\lambda_{\max}^2 + \alpha)}
    \label{eq:mu_cond_tracking}
	\end{align}
    where $\|C\|_1$ represents the maximum absolute column sum of the matrix $C$, while $C_*$ represents the miminum absolute column sum of the matrix $C$. The asymptotic excess-risk at node $k$ then satisfies:
	\begin{align}
	    \mathrm{ER}_k(i) &\leq \underset{\textrm{Steady-state term}}{\underbrace{\frac{\|C\|_1^2 \sigma_v^2 \lambda_{\max}}{4 \lambda_{\min} C_*} \mu}} + \underset{\textrm{Tracking term}}{\underbrace{\frac{\Tr(Q) \lambda_{\max}}{4 \lambda_{\min} C_*} \mu^{-1}}} + \frac{M \lambda_{\max}}{2} \Tr(Q)
		\label{eq:upper_bound_tracking}
	\end{align}
	for all $k = 1,\ldots,N$. Since all nodes satisfy this bound, the network excess-risk, $\mathrm{ER}(i)$, is also asymptotically bounded by the right-hand-side of \eqref{eq:upper_bound_tracking}.
\end{theorem}
\begin{proof}
To show that the asymptotic excess-risk at node $k$ is bounded, we observe that the excess-risk is asymptotically approximated by the weighted mean-square-error \eqref{eq:WMSE_pred} with weight matrix $T_k$ given in \eqref{eq:T_k_infty}:
\begin{align}
\mathrm{ER}_k(i) &\approx \E\|\tilde{\w}_{k,i}^p\|_{T_k}^2 = \E\|\tilde{\w}_{k,i}^p\|_{\frac{1}{2} \nabla^2_w J_k(\w^o_i)}^2
\end{align}
Using \eqref{eq:relationship_prediction_filtering}, we see that the excess-risk can be written in terms of the filtering error:
\begin{align}
	\mathrm{ER}_k(i) &\leq \E\|\tilde{\w}_{k,i}^f\|_{\frac{1}{2} \nabla^2_w J_k(w^o_i)}^2 + \Tr(Q T_k) \nonumber\\
	&\leq \frac{\lambda_{\max}}{2} \E\|\tilde{\w}_{k,i}^f\|^2 + \Tr(Q T_k)
	\label{eq:ER_filtering}
\end{align}
where $\tilde{\w}_{k,i}^f \triangleq \w^o_i - \w_{k,i}$ and the inequality is a result of Assumption \ref{ass:HessianAssumption}. We can use Assumption \ref{ass:HessianAssumption} to verify that $\Tr(Q T_k)$ is also bounded since:
\begin{align*}
	\Tr(Q T_k) &= \sum_{m=1}^M \sum_{n=1}^M Q_{m n} T_{k,m n} \\
	&\overset{(a)}{\leq} \sqrt{\left(\sum_{m=1}^M \sum_{n=1}^M Q_{m n}^2\right) \left(\sum_{m=1}^M \sum_{n=1}^M T_{k, m n}^2\right)}\\
	&= \sqrt{\Tr\left(Q^2\right) \Tr\left(T_{k}^2\right)}\\
	&\overset{(b)}{=} \sqrt{\Tr(U \Omega^2 U^\T) \Tr(V \Pi^2 V^\T)}\\
	&= \sqrt{\left(\sum_{m=1}^M \omega_m^2\right) \left(\sum_{m=1}^M \pi_m^2\right)}\\
	&\overset{(c)}{\leq} \sqrt{\left(\sum_{m=1}^M \omega_m\right)^2 \left(\sum_{m=1}^M \pi_m\right)^2}\\
	&= \sqrt{\left(\Tr(Q)\right)^2 \left(\Tr(T_k)\right)^2}\\
	&\overset{(d)}{\leq} \frac{M \lambda_{\max}}{2} \Tr(Q) 
\end{align*}
where step $(a)$ is due the Cauchy-Schwarz inequality, step $(b)$ is due to the introduction of the eigenvalue decompositions $Q=U \Omega U^\T$ and $T_k = V \Pi V^\T$, where $\Omega = \textrm{diag}\{\omega_1,\ldots,\omega_M\}$ and $\Pi = \textrm{diag}\{\pi_1,\ldots,\pi_M\}$ are the  non-negative eigenvalues of the symmetric matrices $Q$ and $T_k$, respectively. Step $(c)$ is due to $Q$ and $T_k$ being non-negative definite, and step $(d)$ is due to Assumption \ref{ass:HessianAssumption}. This means that the excess-risk at node $k$ \eqref{eq:ER_filtering} can be upper-bounded by
\begin{align}
\mathrm{ER}_k(i) \leq \frac{\lambda_{\max}}{2} \E\|\tilde{\w}_{k,i}^f\|^2 + \frac{M \lambda_{\max} }{2} \Tr(Q)
\label{eq:ER_filtering2}
\end{align}
To bound the filtering error $\E\|\tilde{\w}_{k,i}^f\|^2$, from \ref{app:proof_convergence}, we have the scalar recursion \eqref{eq:intermediate_rec_ineq}:
\begin{align}
    \|\mathcal{W}_i\|_\infty &\leq \beta^i \|\mathcal{W}_{0}\|_\infty +
							 \left(\|C\|_1^2 \sigma_v^2 \mu^2 + \Tr(Q)\right) \sum_{j=0}^{i-1} \beta^j
    \label{eq:intermediate_rec_ineq1}
\end{align}
where $\|x\|_\infty$ denotes the maximum absolute entry of a vector $x$ and 
\begin{align}
	\mathcal{W}_i &\triangleq \big[\E\|\tilde{\w}_{1,i}^f\|^2,\dots,\E\|\tilde{\w}_{N,i}^f\|^2\big]^\T\\
	\beta &\triangleq 1-2 \mu \lambda_{\min} C_* + \mu^2 (\lambda_{\max}^2 + \alpha) \|C\|_1^2
\end{align}
Notice that when the constant step-size $\mu$ satisfies \eqref{eq:mu_cond_tracking}, we have that $\beta < 1$. Therefore, we can evaluate the limit of the geometric series in the second term of \eqref{eq:intermediate_rec_ineq1} as
\begin{align}
    \lim_{i\rightarrow\infty} \left(\|C\|_1^2 \sigma_v^2 \mu^2 + \Tr(Q)\right) \sum_{j=0}^{i-1} \beta^j = \frac{\|C\|_1^2 \sigma_v^2 \mu^2 + \Tr(Q)}{1-\beta}
\end{align}
Additionally, the limit of the first term on the right-hand-side of \eqref{eq:intermediate_rec_ineq1} will be zero since $\beta < 1$. Therefore, we have that
\begin{align}
    &\limsup_{i\rightarrow\infty} \|\mathcal{W}_i\|_\infty \leq \frac{\|C\|_1^2 \sigma_v^2 \mu^2 + \Tr(Q)}{1-\beta} \nonumber\\
    %&= \frac{\|C\|_1^2 \sigma_v^2 \mu^2}{2 \mu \lambda_{\min} C_* - \mu^2 (\lambda_{\max}^2 + \alpha)\|C\|_1^2} + \frac{\Tr(Q)}{2 \mu \lambda_{\min} C_* - \mu^2 (\lambda_{\max}^2 + \alpha)\|C\|_1^2}\nonumber\\
    &= \frac{\|C\|_1^2 \sigma_v^2 \mu}{2 \lambda_{\min} C_* - \mu (\lambda_{\max}^2 \!+\! \alpha)\|C\|_1^2} + \frac{\Tr(Q)}{2 \mu \lambda_{\min} C_* \!-\! \mu^2 (\lambda_{\max}^2 \!+\! \alpha)\|C\|_1^2}
    \label{eq:limsup_ex1}
\end{align}
For sufficiently small step-sizes, the denominator of the first and second terms of \eqref{eq:limsup_ex1} can be respectively approximated by
\begin{align}
    2 \lambda_{\min} C_* - \mu (\lambda_{\max}^2 + \alpha)\|C\|_1^2 &\approx 2 \lambda_{\min} C_*\\
    2 \mu \lambda_{\min} C_* - \mu^2 (\lambda_{\max}^2 + \alpha)\|C\|_1^2 &\approx 2 \mu \lambda_{\min} C_*
\end{align}
Therefore, we conclude that \eqref{eq:limsup_ex1} can be approximated for small step-sizes by
\begin{align}
\limsup_{i\rightarrow\infty} \|\mathcal{W}_i\|_\infty &\leq \frac{\|C\|_1^2 \sigma_v^2 }{2 \lambda_{\min} C_*} \mu + \frac{\Tr(Q)}{2 \lambda_{\min} C_*} \mu^{-1}
\end{align}
Noting the relationship between excess-risk and the mean square error in \eqref{eq:ER_filtering2}, we have that the excess-risk at node $k$ is bounded by
\begin{align}
        \mathrm{ER}_k(i) &\leq \frac{\|C\|_1^2 \sigma_v^2 \lambda_{\max}}{4 \lambda_{\min} C_*} \mu + \frac{\Tr(Q) \lambda_{\max}}{4 \lambda_{\min} C_*} \mu^{-1} + \frac{M \lambda_{\max}}{2} \Tr(Q)
\end{align}
and therefore the network excess-risk $\mathrm{ER}(i)$ satisfies this bound as well for sufficiently large $i$ and small $\mu$.
\end{proof}
Consider the case where $C=I_N$. We observe from \eqref{eq:upper_bound_tracking} that a trade-off exists between the steady-state performance of the algorithm and its tracking performance. The bound consists of the sum of the steady state excess-risk \eqref{eq:RegretBound} derived for stationary environments and a term that depends on $\mu^{-1}$ and which arises as a result of the random-walk model noise $\q_i$. To decrease the steady-state error, we would need to use a smaller step-size, which affects the tracking performance adversely. Figure~\ref{fig:tracking_tradeoff} illustrates this trade-off. In the figure, $\mu^o$ indicates the optimal choice for the step-size in order to minimize the bound on the right-hand-side of \eqref{eq:upper_bound_tracking}. The figure gives insight into the fact that a small step-size will improve the steady-state performance when the environment is stationary, but will harm the tracking ability of the algorithm when the environment is non-stationary. We conclude that the asymptotic network excess-risk \eqref{eq:ER_dist} remains upper-bounded by a constant, even when the optimizer changes according to a random-walk. That is, even as the variance of the random process generating $\w_i^o$ grows indefinitely, the excess-risk at each node remains bounded.

In order to illustrate the application of the result in the context of machine learning, we consider a linear binary classification problem where the task is to find a hyper-plane (through the origin) that best separates features from two classes according to some cost function (such as the logistic regression cost in \eqref{eq:log-loss}). Since the hyper-plane is fixed at the origin, the task is to find the best rotation of the hyper-plane to separate the data. Consider now that the distribution from which the feature vectors arise is time varying and as a result the optimal hyper-plane must rotate accordingly --- see Fig.~\ref{fig:rotating_hyperplane}. Our analysis shows that the diffusion algorithm can track the random-walk rotating hyper-plane proposed in \cite{Jenn2010} and remain within a constant excess-risk on average for any strongly-convex cost function used that satisfies Assumption \ref{ass:HessianAssumption}.
\begin{figure}
\centering
	\includegraphics[width=0.5\textwidth]{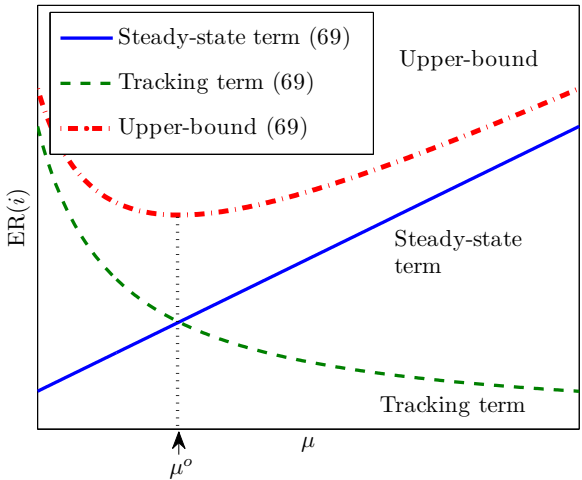}
	\caption{Trade-off between tracking performance and steady-state excess-risk. The scalar $\mu^o$ indicates the optimal choice for the step-size in order to minimize the bound on the excess-risk.}
	\label{fig:tracking_tradeoff}
\end{figure}
\begin{figure*}
	\centering
		\includegraphics[width=0.65\textwidth]{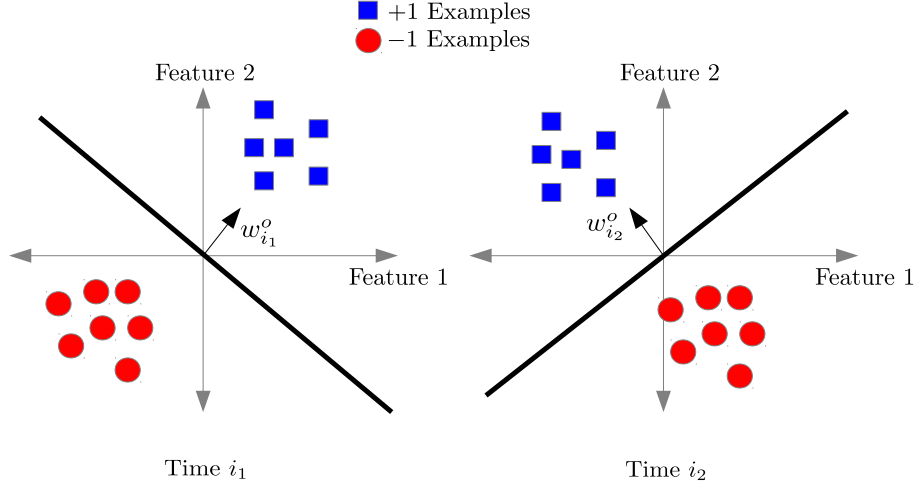}
		\caption{A rotating hyper-plane in 2D that adjusts to separate data from two classes $\{+1,-1\}$. $w^o_i$ indicates the optimal normal vector of the hyper-plane.}
		\label{fig:rotating_hyperplane}
\end{figure*}
%\vspace*{-1em}
\section{Simulation Results}
	\subsection{Stationary Environments}
	   In this section, we test the distributed diffusion strategy \eqref{eq:C1}-\eqref{eq:C2} on three stationary datasets:
        \begin{itemize}
            \item The `alpha' dataset \cite{alphaDataset}.
            \item The `a9a' dataset \cite{covtype}.
            \item The `webspam' (unigram) dataset \cite{covtype}.
        \end{itemize}
       Each set deals with a binary classification problem. The dataset properties are compiled in Table \ref{tbl:datasets_stationary}. We split the data evenly across the nodes with the step-size chosen so that it is possible to observe the steady-state behavior. Unfortunately, since some of the datasets are relatively small (once divided over the nodes), this means that the step-size chosen needs to be relatively large. The analysis we have for the approximate steady-state expression in Theorem \ref{thm:SS_stationary} assumes the use of small step-sizes, so we expect to see a better match between theory and simulation if the data sets were larger and the step-sizes were smaller---see Figs.~\ref{fig:ER_webspam_large_mu}-\ref{fig:ER_webspam_small_mu} further ahead. Better matches will occur when smaller step-sizes are used \cite{Jianshu_common_wo,Xiaochuan}.
       \begin{table*}
            \begin{centering}
            \caption{Properties of datasets used for performance evaluation and the problem parameters associated with the datasets.}
            \label{tbl:datasets_stationary}
            \begin{tabular}{|c|c|c|c|c|c|c|}
            \hline
            Dataset & Instances & Attributes ($M$) & $\rho$ & $\mu$ & $N$ & Experiments\tabularnewline
            \hline
            \hline
            \textbf{alpha} & 500000 & 500 & 5 & 0.0001 & 20 & 20\tabularnewline
            \hline
            \textbf{a9a} & 32561 & 123 & 5 & 0.02 & 8 & 100\tabularnewline
            \hline
            \textbf{webspam} & 350000 & 254 & 5 & 0.0025/0.001 & 40 & 50\tabularnewline
            \hline
            \end{tabular}
            \par\end{centering}

        \end{table*}
        \begin{figure*}[H]
        \end{figure*}
        We perform regularized logistic regression \eqref{eq:log-loss} on the dataset in real-time and evaluate the network excess-risk defined in \eqref{eq:ER_dist} using the ATC, CTA, and the non-cooperative algorithms described by \eqref{eq:ATC2}, \eqref{eq:CTA_C=I}, and \eqref{eq:no_coop}, respectively. For the ATC and CTA algorithms, we set the gradient combination matrix $C=I_N$ so that the nodes do not exchange their gradient vectors. In addition, we compare the performance of our algorithm to the centralized full gradient (CFG) algorithm that has access to all data samples from all $N$ nodes at every iteration:
        \begin{align}
            \w_{\textrm{CFG},i} = \w_{\textrm{CFG},i-1} - \frac{\mu}{N} \sum_{k=1}^N \widehat{\nabla}_w J_{k}(w_{\textrm{CFG},i-1}) \quad\quad\textrm{(CFG)}
            \label{eq:CFG}
        \end{align}
        The CFG algorithm averages the gradients from all nodes and moves against the average gradient direction. We also compare against the semi-distributed algorithm from \cite{zinkevich_parallelized} where each node executes stochastic gradient descent up to some time horizon $i$ and then the nodes  transmit their estimates $\w_{k,i}$ to a central processor that averages all estimates:
        \begin{align}
            \w_{\textrm{THA},k,i} \triangleq \frac{1}{N} \sum_{k=1}^N \w_{k,i}\quad\quad\textrm{(time-horizon averaging)}
            \label{eq:zink}
        \end{align}
        Notice that \eqref{eq:zink} requires some time horizon $i$ to be known and requires either some central server to average the estimates and redistribute the average \eqref{eq:zink} back to the nodes or the use of some iterative consensus scheme \cite{DeGroot}. In order to compare our algorithm to that of \cite{zinkevich_parallelized}, we assume that the averaging occurs at every step of the algorithm (we only evaluate the excess-risk at the central processor, and do not communicate the average back to the nodes since the nodes' iterations do not depend on the averaged estimates). Finally, we also simulate algorithm \eqref{eq:consensus} from \cite{yan} using a constant step-size. The same step-size is used for all algorithms. For the combination matrix $A$, we utilize the Metropolis rule \cite{federico} to generate the coefficients:
        \begin{align}
        	a_{\ell k} = \begin{cases}
        			\min\left(\frac{1}{|\mathcal{N}_\ell| -1}, \frac{1}{|\mathcal{N}_k|-1}\right), & \ell \in \mathcal{N}_k, \ell \neq k\\
        			1 - \sum_{j = 1}^N a_{j k},& \ell = k\\
        			0, & \textrm{otherwise}
        	\end{cases}
        	\label{eq:Metropolis}
        \end{align}
        The Metropolis weighting matrix $A$ generated using \eqref{eq:Metropolis} is doubly stochastic. The loss function that each node utilizes is the regularized log-loss:
        \begin{align}
            Q(w,\h_i,\y_i) &\triangleq \frac{\rho}{2} \|w\|^2 + \log(1+e^{-\y_i \h_i^\T w})
            \label{eq:Q_sim}
        \end{align}
        where $\h_i$ indicates the feature vector and $\y_i$ indicates the true label ($\pm 1$). In this case, the data $\x_{k,i}$ in \eqref{eq:J_k_def} are defined as $\x_{k,i} \triangleq \{\h_{k,i},\y_{k,i}\}$. The risk function is the expectation of $Q(\cdot)$ over the inputs $\h_i$ and $\y_i$. In each experiment, a number $N$ of nodes are used to distribute the classifier learning task as listed in Table \ref{tbl:datasets_stationary}. A batch optimization, where all samples from the full dataset are available to the learner, was used in order to compute $w^o$. This optimization was conducted using the \texttt{LIBLINEAR} \cite{liblinear} library. The theoretical curves are computed using the simplified expressions derived in \cite{Xiaochuan,MLSP2}:
\begin{align}
	\mathrm{ER}_k(i) \approx \frac{\mu \Tr(R_{v,k})}{4 N}
\end{align}        
where $R_{v,k} \triangleq \E\{\v_{k,i}(w^o) \v_{k,i}(w^o)^\T\}$. Fig.~\ref{fig:Stationary} shows the excess-risk learning curves for the different algorithms and different datasets. We observe that the ATC algorithm outperforms the CTA algorithm and the non-cooperative algorithm (as established by Theorem \ref{thm:ATC_CTA_Ind}) as well as the consensus-type algorithm \eqref{eq:consensus} from \cite{yan} when the same constant step-size is used. We also observe from Figs.~\ref{fig:ER_webspam_large_mu} and \ref{fig:ER_webspam_small_mu} that as the step-size decreases, the excess-risk also decreases. This fact is in agreement with our analysis in Theorem \ref{thm:convergence}. We notice that the time-horizon averaging algorithm from \cite{zinkevich_parallelized} is close in performance to the ATC diffusion algorithm. The algorithm from \cite{zinkevich_parallelized}, however, requires global communication at every iteration and is not a distributed solution as is the case with diffusion strategies.
%\begin{figure*}
%    \centering
%    \subfloat[][Excess-risk for `alpha' dataset]{
%    	    \includegraphics[width=0.45\textwidth]{simulation_ER_alpha_rho_5_alpha_0001_v4.eps}
%    	    \label{fig:ER_alpha}
%    	}
%    \quad
%    \subfloat[][Excess-risk for `a9a' dataset]{
%    	    \includegraphics[width=0.45\textwidth]{simulation_ER_a9a_rho_5_mu_02_N_8_ex_100_v3.eps}
%    	    \label{fig:ER_a9a}
%    	}\\
%    \subfloat[][Excess-risk for `cod-rna' dataset ($\mu = 0.00002$)]{
%    	    \includegraphics[width=0.45\textwidth]{simulation_ER_codrna_rho_5_mu_00002_N_16_ex_500_v3.eps}
%    	    \label{fig:ER_codrna_large_mu}
%    	} \quad
%    \subfloat[][Excess-risk for `cod-rna' dataset ($\mu = 0.00001$)]{
%    	    \includegraphics[width=0.45\textwidth]{simulation_ER_codrna_rho_5_mu_00001_N_16_ex_500_v3.eps}
%    	    \label{fig:ER_codrna_small_mu}
%    	}\\
%    \subfloat[][Excess-risk for `webspam' dataset ($\mu = 0.0025$)]{
%    	    \includegraphics[width=0.45\textwidth]{simulation_ER_webspam_rho_5_mu_0025_N_40_ex_50_v3.eps}
%    	    \label{fig:ER_webspam_large_mu}
%    	} \quad
%    \subfloat[][Excess-risk for `webspam' dataset ($\mu = 0.001$)]{
%    	    \includegraphics[width=0.45\textwidth]{simulation_ER_webspam_rho_5_mu_001_N_40_ex_50_v4.eps}
%    	    \label{fig:ER_webspam_small_mu}
%    	}
%    \caption{Excess-risk learning curves for different stationary datasets (continued on the next page).}
%	\label{fig:Stationary}
%\end{figure*}

\begin{figure*}
    \centering
    \subfloat[][Excess-risk for `alpha' dataset]{
    	    \includegraphics[width=0.6\textwidth]{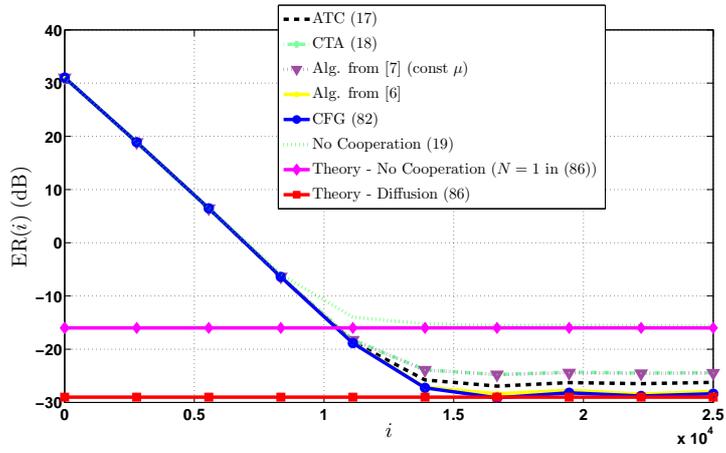}
    	    \label{fig:ER_alpha}
    	}
    \\
    \subfloat[][Excess-risk for `a9a' dataset]{
    	    \includegraphics[width=0.6\textwidth]{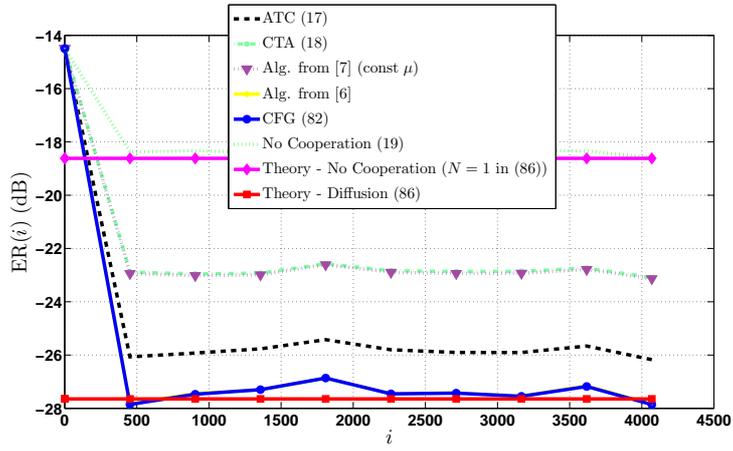}
    	    \label{fig:ER_a9a}
    	}\\
    	\subfloat[][Excess-risk for `webspam' dataset ($\mu = 0.0025$)]{
    	    \includegraphics[width=0.6\textwidth]{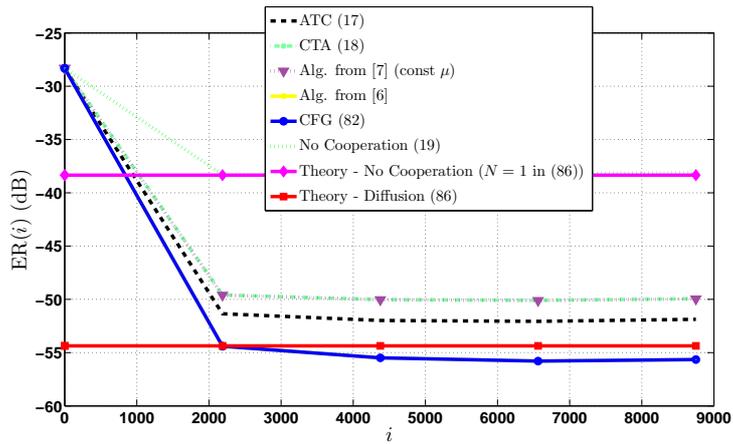}
    	    \label{fig:ER_webspam_large_mu}
    	}
    \caption{Excess-risk learning curves for different stationary datasets (continued on the next page).}
%	\label{fig:Stationary}
\end{figure*}
\begin{figure*}
    \ContinuedFloat
    \centering
    \subfloat[][Excess-risk for `webspam' dataset ($\mu = 0.001$)]{
    	    \includegraphics[width=0.6\textwidth]{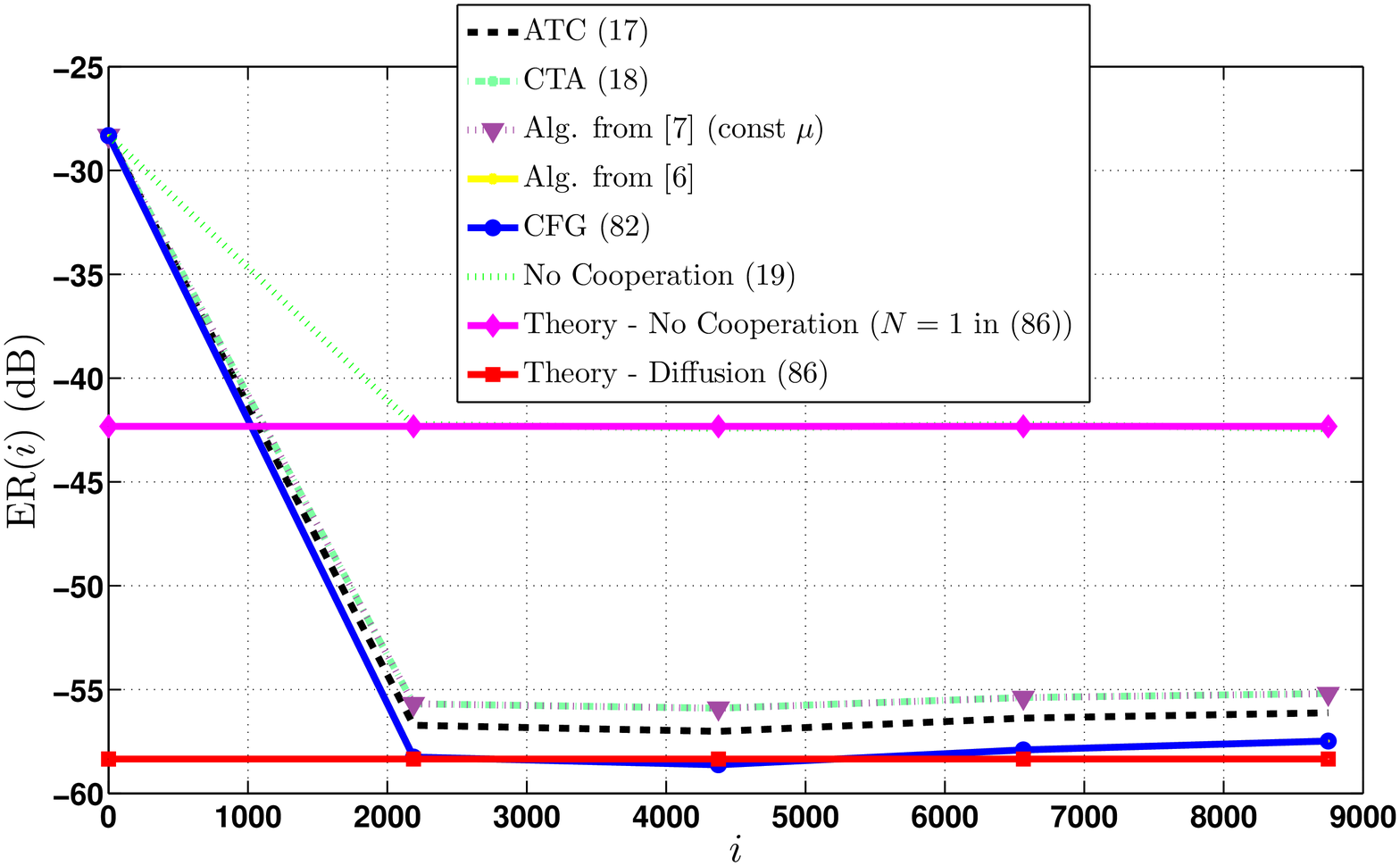}
    	    \label{fig:ER_webspam_small_mu}
    	}
    \caption{Excess-risk learning curves for different stationary datasets (continued from the previous page).}
    \label{fig:Stationary}
\end{figure*}

        In order to evaluate the performance of the actual classifier output by the algorithms, we plot the receiver operating characteristic (ROC) curves in Fig.~\ref{fig:Stationary_ROC}. The classifier for each of the algorithms is computed using:
        \begin{align}
            \hat{y}_i \triangleq \mathrm{sign}(h_i^\T w - b)
            \label{eq:classifier}
        \end{align}
        by sweeping the bias $b$. In Fig.~\ref{fig:Stationary_ROC}, $P_{\mathrm{D}}$ indicates the probability of detection while $P_{\mathrm{FA}}$ indicates the probability of false alarm. Notice that the curve for the ATC algorithm is very close to that of the CFG algorithm and the algorithm from \cite{zinkevich_parallelized} while the ATC algorithm is \emph{fully} distributed. The CTA and consensus algorithm from \cite{yan} perform worse than the ATC algorithm. We also see a clear performance improvement over the non-cooperative algorithm. Finally, as the step-size decreases for the `webspam' dataset, we see that the diffusion algorithm tends to improve in performance and get closer to the centralized batch processing solution. The batch processing curve is computed by using $w^o$ as the separating hyperplane in \eqref{eq:classifier}.
\begin{figure*}
    \centering
        \subfloat[][ROC curve for `alpha' dataset]{
        	    \includegraphics[width=0.45\textwidth]{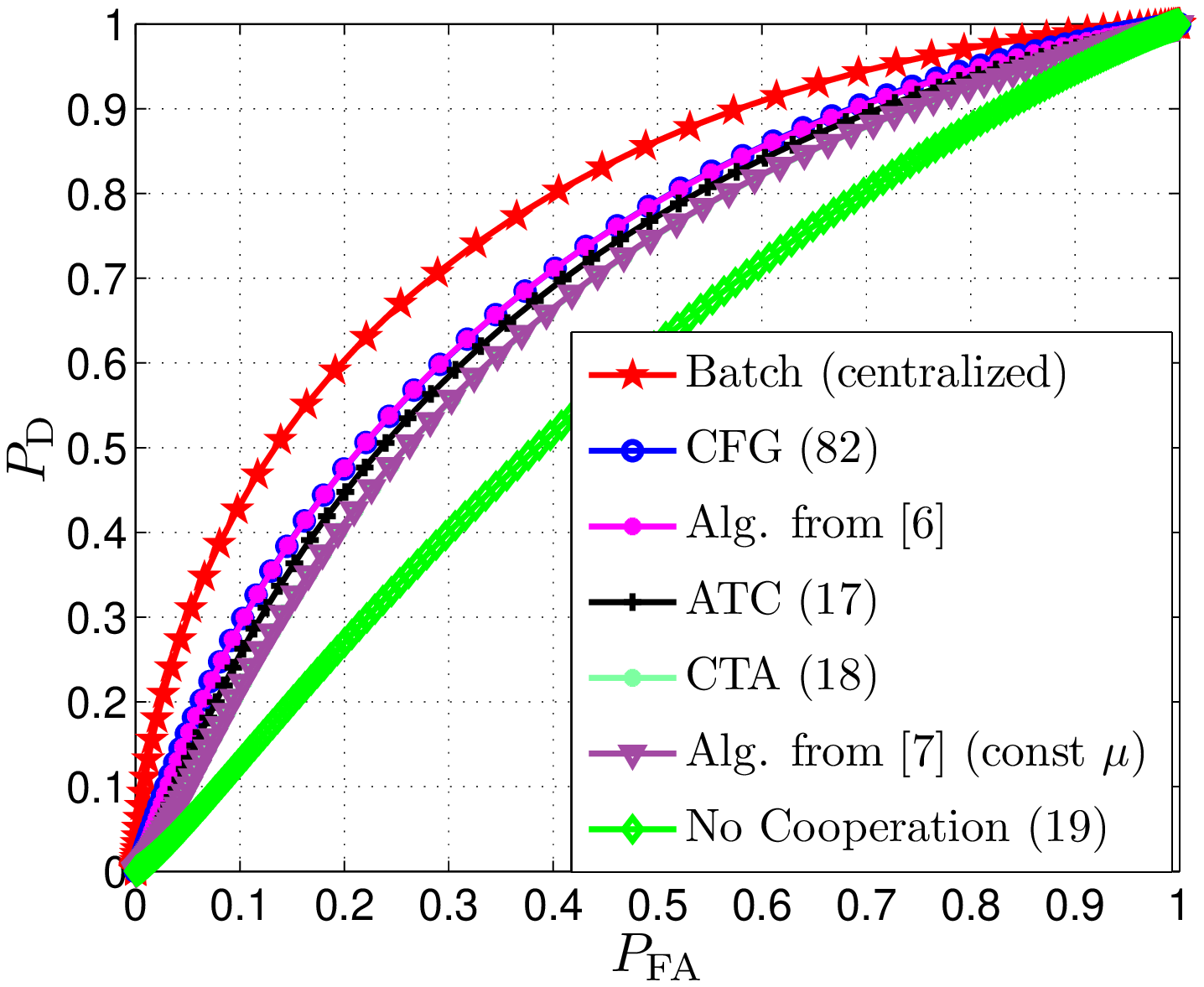}
        	    \label{fig:ROC_alpha}
        	}
        \quad
        \subfloat[][ROC curve for `a9a' dataset]{
        	    \includegraphics[width=0.45\textwidth]{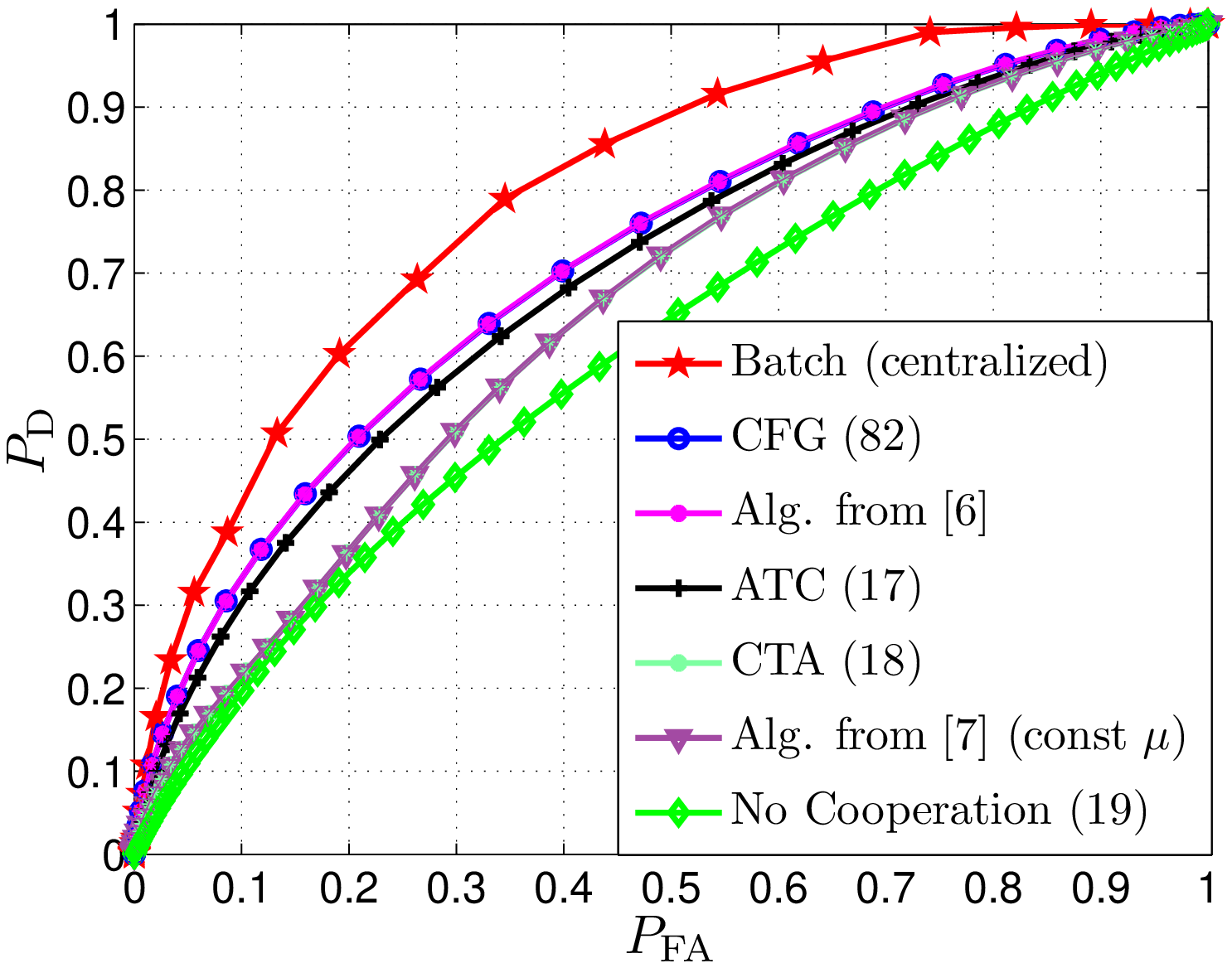}
        	    \label{fig:ROC_a9a}
        }\\
%        \subfloat[][ROC curve for `cod-rna' dataset ($\mu = 0.00001$)]{
%    	    \includegraphics[width=0.45\textwidth]{ROC_codrna_rho_5_mu_00001_N_10_ex_500_v2.eps}
%    	    \label{fig:ROC_codrna_small_mu}
%    	}\\
        \subfloat[][ROC curve for `webspam' dataset ($\mu = 0.0025$)]{
    	    \includegraphics[width=0.45\textwidth]{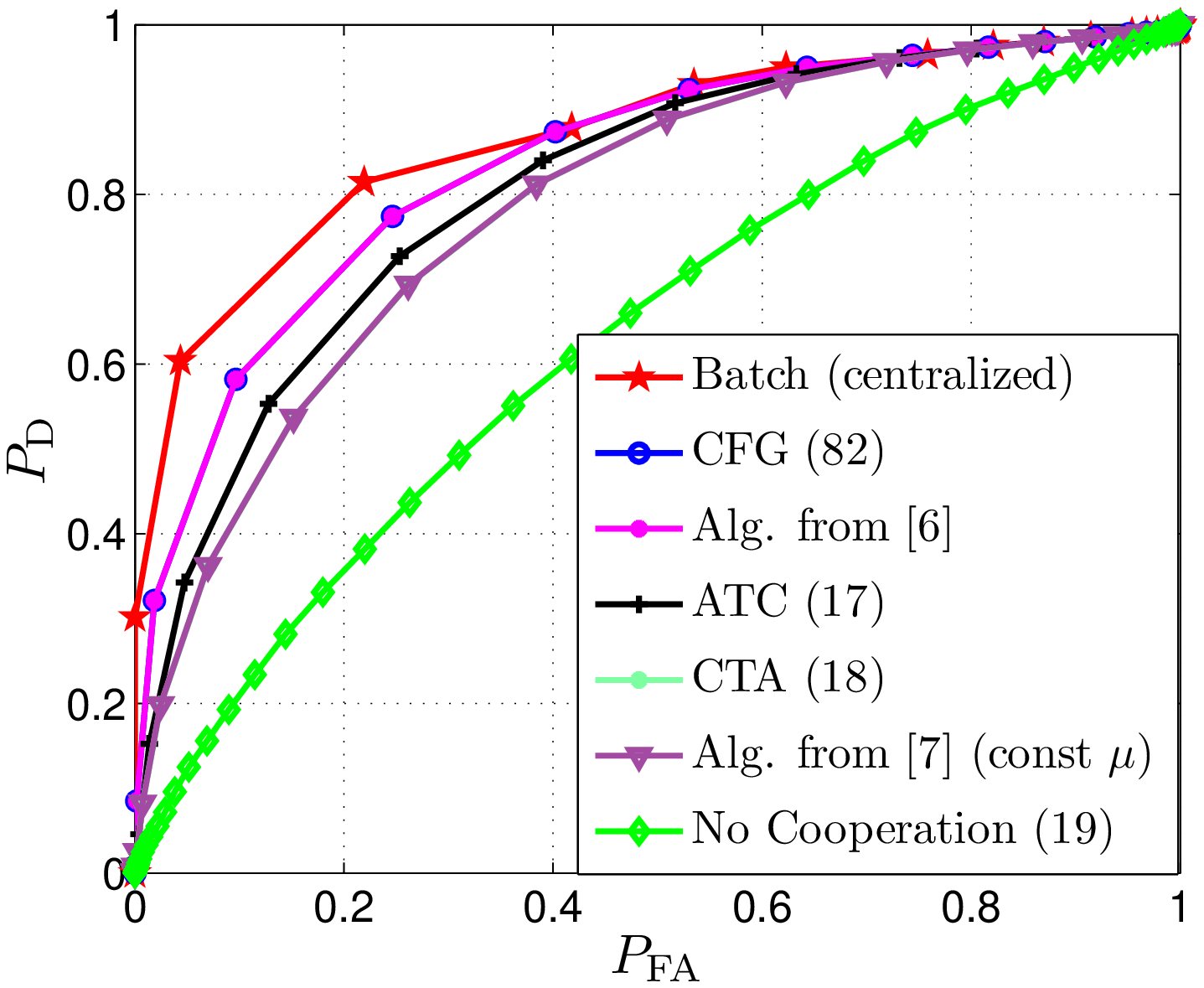}
    	    \label{fig:ROC_webspam_large_mu}
    		}
        \quad
        \subfloat[][ROC curve for `webspam' dataset ($\mu = 0.001$)]{
    	    \includegraphics[width=0.45\textwidth]{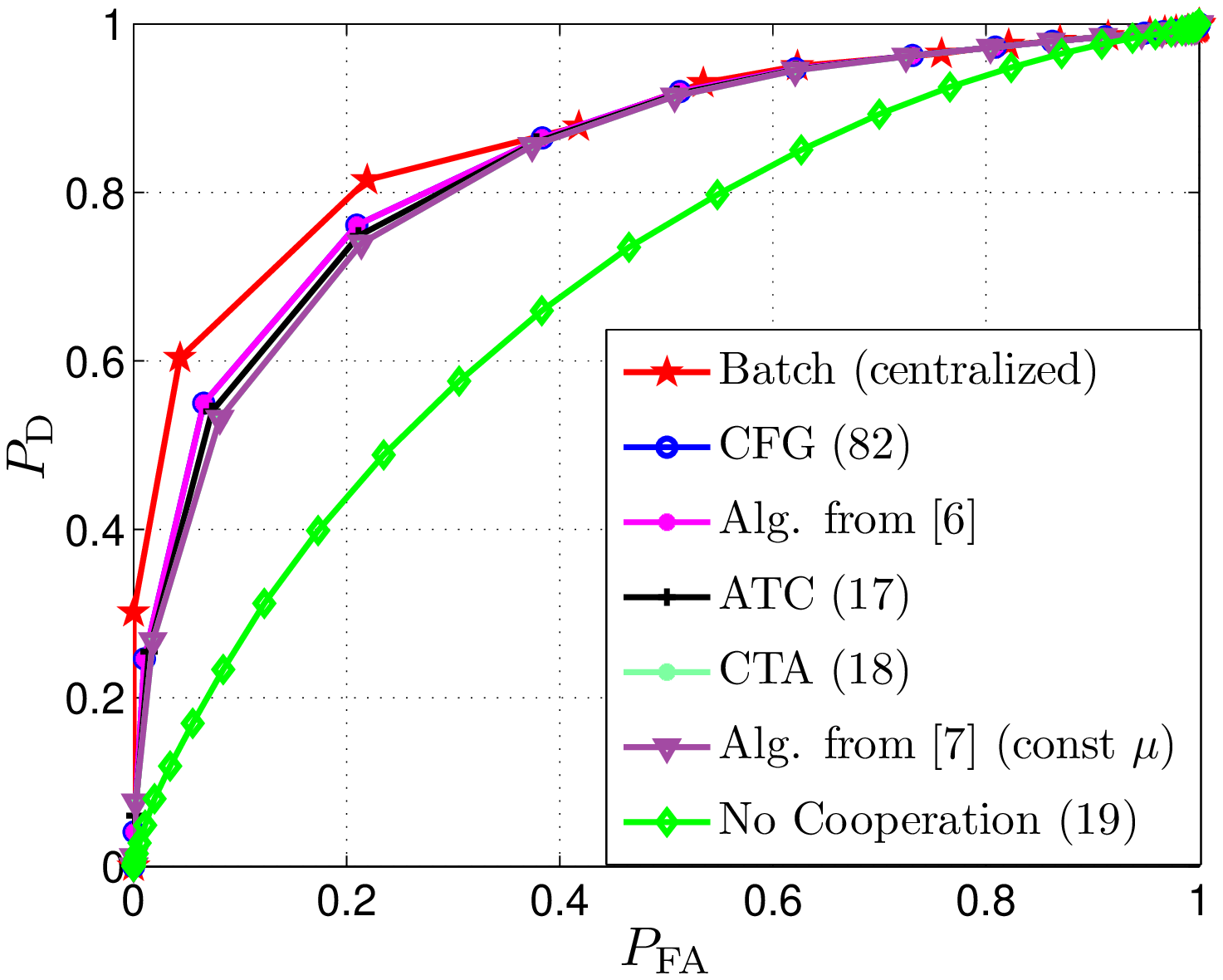}
    	    \label{fig:ROC_webspam_small_mu}
    	}
    \caption{ROC curves for different stationary datasets.}
	\label{fig:Stationary_ROC}
\end{figure*}

\subsection{Non-Stationary Environments}
    \subsubsection{Random Walk Rotating Hyperplane - Gradual Concept Drift}
    In this section, we simulate a scenario where $\w^o_i$ is a random walk. We do so to illustrate the analysis in Theorem \ref{thm:tracking} and to simulate the behavior of the algorithms under gradual concept drifts. In the next section, we will simulate instant concept drifts. In order to clarify the presentation of the results, we concentrate in this section on the ATC algorithm and the algorithm from \cite{yan} only since we have already established in the last section that the ATC algorithm outperforms CTA and non-cooperation. We study the algorithm from \cite{yan} when the step-size decays with time. This allows us to highlight the importance of utilizing constant step-sizes in non-stationary environments.
    We generate data for two classes $\{+1,-1\}$ with Gaussian distributions $\mathcal{N}(\m_i,I_2)$ and $\mathcal{N}(-\m_i,I_2)$ respectively where $\m_i$ is the mean of the $+1$ distribution at time $i$. We let $\m_i$ be a random walk with increments that are Gaussian with zero mean and covariance $0.01 I_2$. We compute $w^o_i$ at every iteration based on all the data in the network using the \texttt{LIBLINEAR} library \cite{liblinear}. Each of the $N=200$ nodes receives one sample per iteration. The Metropolis weights \eqref{eq:Metropolis} are used to combine the estimates for the ATC algorithm and the algorithm from \cite{yan}. An amount of $10\%$ label noise was also added to the dataset. We set the step-size to $\mu = 0.005$ and $\rho = 0.01$ for the loss function in \eqref{eq:Q_sim}. We use the classifier in \eqref{eq:classifier} to obtain the classifier accuracy in Fig.~\ref{fig:accuracy_markov}, which is defined as:
\begin{align}
	\textrm{Accuracy} = \frac{\textrm{Number of correctly classified samples}}{\textrm{Total number of samples}}
\end{align}
In addition, we plot the excess-risk in Fig.~\ref{fig:ER_markov}. We observe that as the target $w^o_i$ changes, the diminishing step-size algorithm from \cite{yan} does not cope with non-stationarity. On the other hand, and as predicted by Theorem \ref{thm:tracking}, the constant step-size algorithm can track these changes.
     \begin{figure*}
	    \centering
	    \subfloat[][Accuracy for the Markov random walk concept drift across time. Larger values are better.]{
	    	    \includegraphics[width=0.6\textwidth]{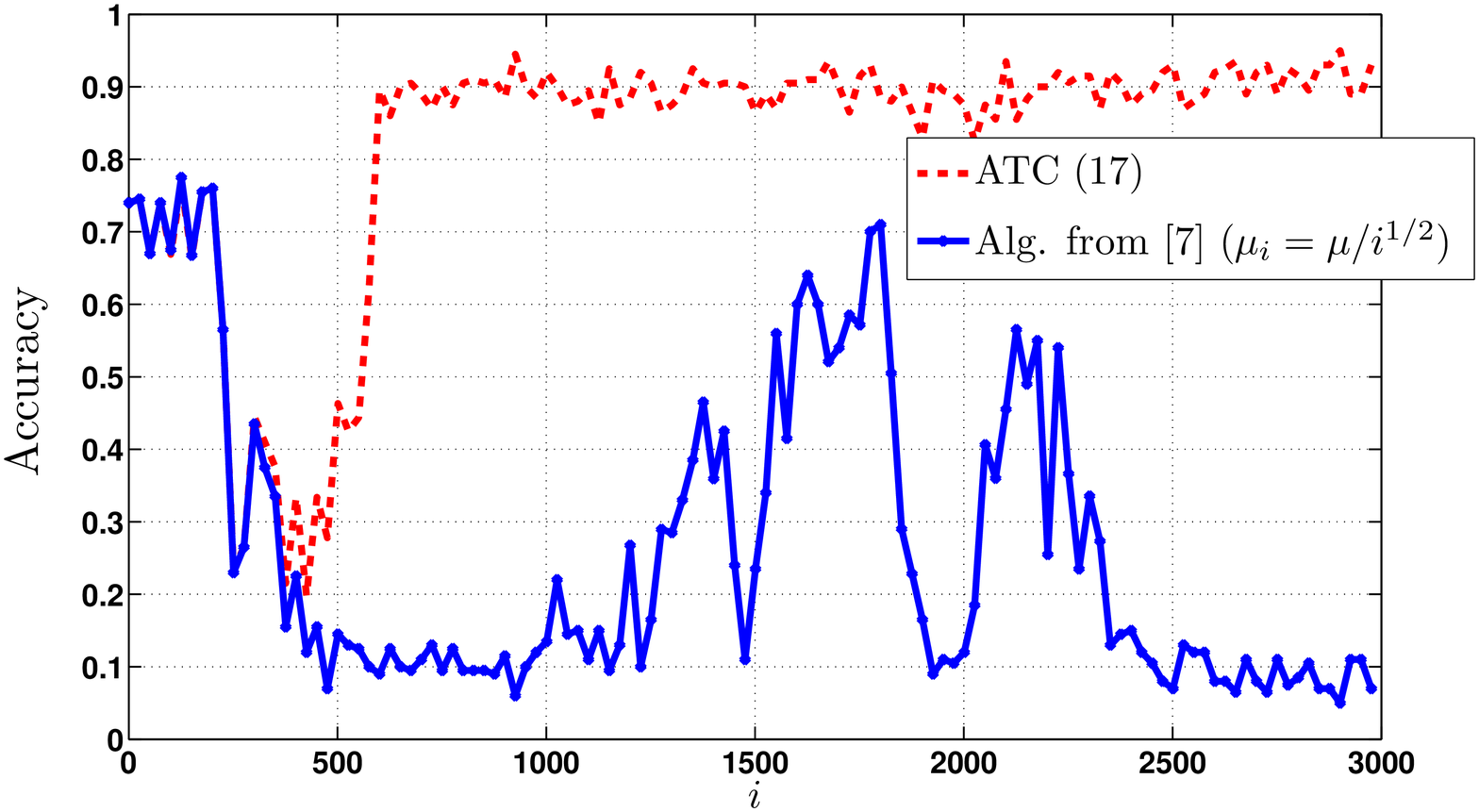}
	    	     \label{fig:accuracy_markov}
	    	}\\
	    \subfloat[][Excess-risk for the Markov random walk concept drift over time. Smaller values are better.]{
	    	    \includegraphics[width=0.6\textwidth]{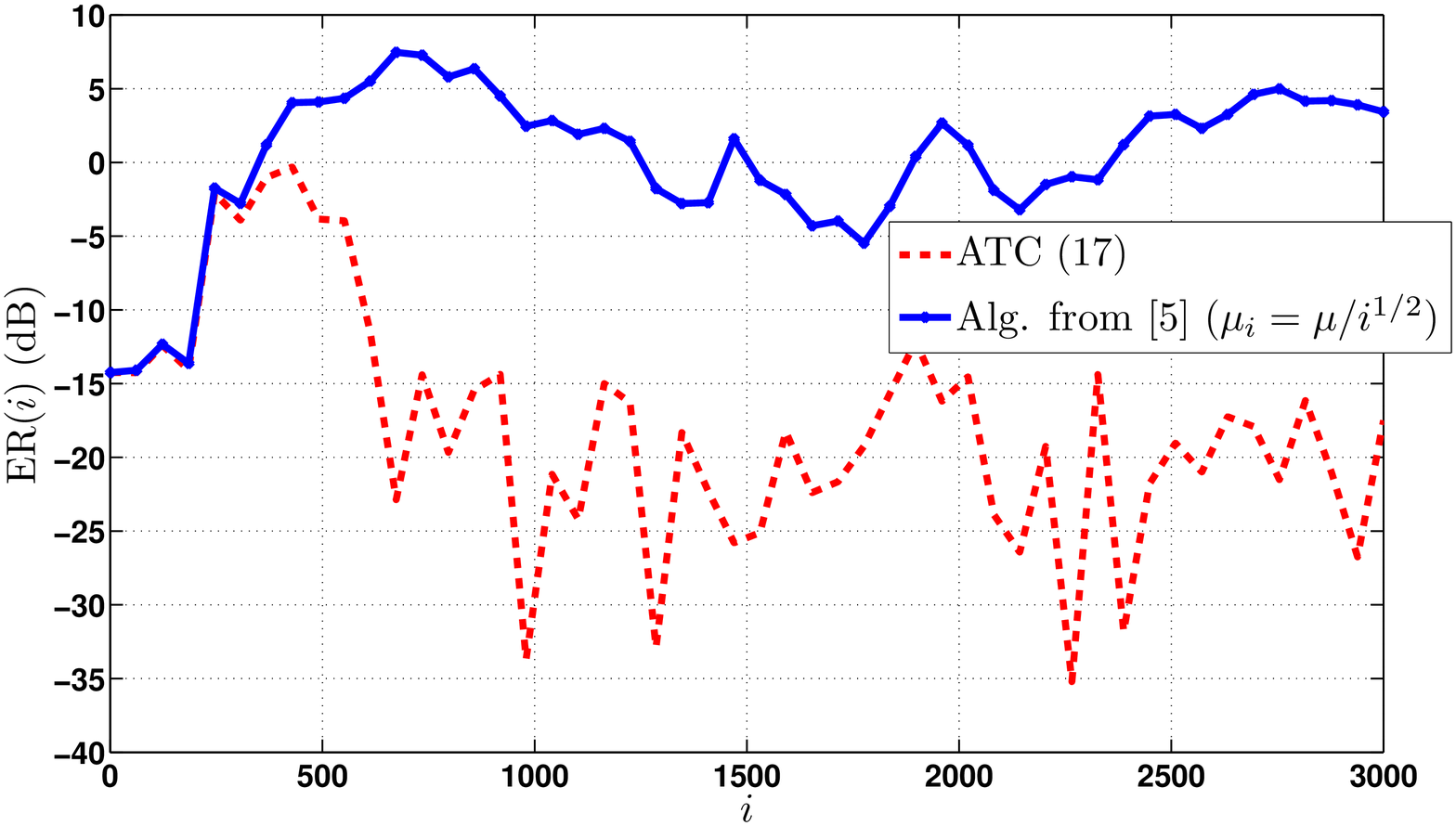}
	    	     \label{fig:ER_markov}
	    	}
	    	\caption{Results for Markov random walk simulation.}
	   		\label{fig:Markov_figs}
	\end{figure*}
    \subsubsection{\texttt{STAGGER} Concepts - Instantaneous Concept Drift}
    In addition to the gradual concept drift simulation in the last section, we also simulate our algorithm on a dataset with instantaneous concept drift. We use the \texttt{STAGGER} dataset \cite{STAGGER,DWM} for this purpose. We simulate a network with $N=125$ nodes. All the nodes experience the concept change simultaneously. As in \cite{DWM}, we define the target concept to be changing over $120$ iterations, in intervals of 40 iterations for each target concept:
    \begin{align}
        y_i \triangleq \begin{cases}
                            (h_{i,1} = 1) \textrm{\ and\ } (h_{i,3} = 0), & 1\leq i \leq 40\\
                            (h_{i,1} = 0) \textrm{\ or\ } (h_{i,2} = 0.5), & 41\leq i \leq 80\\
                            (h_{i,3} = 0.5) \textrm{\ or\ } (h_{i,3} = 1), & 81\leq i \leq 120
                        \end{cases}
    \end{align}
    The labels are then mapped from $\{0,+1\}$ to $\{-1,+1\}$. The above rule can be seen as a numerical representation of the color, shape, and size attributes through the definitions in Table \ref{tbl:STAGGER}.
        
		\begin{table*}
        \caption{Numerical Representation of STAGGER concepts}
        \begin{centering}{\footnotesize{
\begin{tabular}{|c|c|c|c|c|c|c|c|c|c|}
        \hline
        Attribute & \multicolumn{3}{c|}{Color ($x_{i,1}$)} & \multicolumn{3}{c|}{Shape ($x_{i,2}$)} & \multicolumn{3}{c|}{Size ($x_{i,3}$)}\tabularnewline
        \hline
        \hline
        Value & Green & Blue & Red & Triangle & Circle & Rectangle & Small & Medium & Large\tabularnewline
        \hline
        Numerical Representation & 0 & 0.5 & 1 & 0 & 0.5 & 1 & 0 & 0.5 & 1\tabularnewline
        \hline
        \end{tabular}}}
        \par\end{centering}
        \label{tbl:STAGGER}
        \end{table*}    
    An amount of $10\%$ label noise was also added to the dataset at each experiment. The simulation results were averaged over $100$ experiments. A regularization factor of $\rho = 0.1$ was used to optimize the log-loss in \eqref{eq:Q_sim}. The batch optimization was carried out using the \texttt{LIBLINEAR} library \cite{liblinear}. A step-size of $\mu = 0.25$ was used to simulate the constant step-size algorithms (ATC, CTA, non-cooperative, \eqref{eq:consensus}, \eqref{eq:CFG}, and \cite{zinkevich_parallelized}). In addition, we simulate the algorithm from \cite{yan} with a diminishing step-size $\mu_i \triangleq \mu/\sqrt{i}$ to illustrate the necessity of constant step-sizes for non-stationary environments. Figure~\ref{fig:ER_STAGGER} shows the excess-risk performance of the different algorithms on the \texttt{STAGGER} concepts. The constant step-size algorithms continuously track the changing target concept while the diminishing step-size algorithm from \cite{yan} fails to do so due to the diminishing learning rate. Observe that the algorithm from \cite{zinkevich_parallelized} would not know when the concept changed and it would have to implement a change detector in order to allow the central node to poll the information from all the nodes (or to initiate consensus iterations). We also evaluate the ROC curves using \eqref{eq:classifier} associated with the classifier at the last iteration of the target concept. The ROC curves are illustrated in Fig.~\ref{fig:ROC_STAGGER}. The diminishing step-size algorithm is not helpful in detecting the second concept since it is below the chance line ($P_{\mathrm{D}} = P_{\mathrm{FA}}$). In addition, we still notice that the ATC algorithm outperforms the other fully distributed approaches (non-cooperative, CTA, and \eqref{eq:consensus}) and is close to the batch solution. Metropolis weights \eqref{eq:Metropolis} are used for the combination matrix for the distributed algorithms.
    \begin{figure*}
    	\centering
    	\subfloat[][Excess-risk performance of the algorithms on the time-varying \texttt{STAGGER} concepts.]{
	    	    \includegraphics[width=0.6\textwidth]{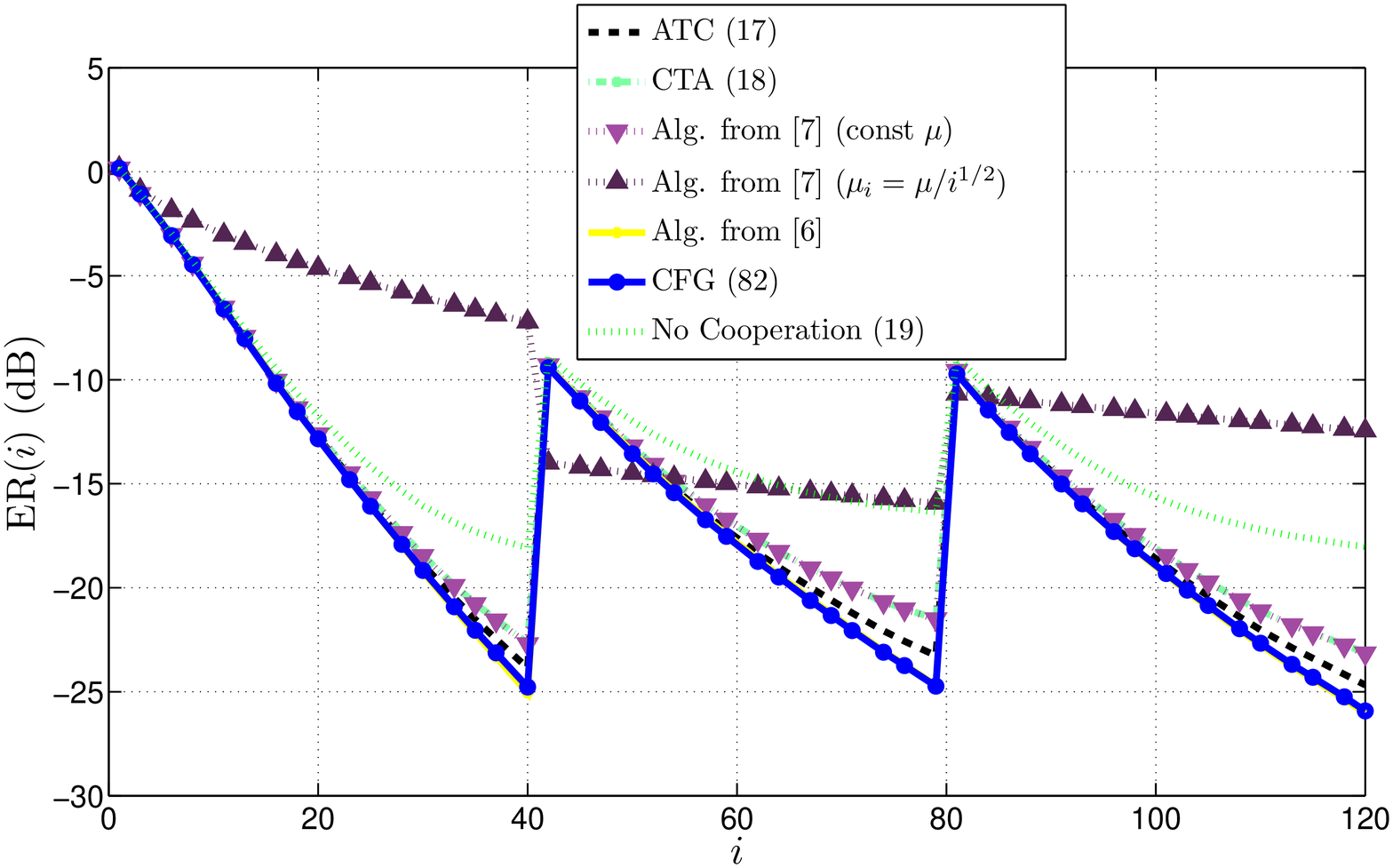}
	    	     \label{fig:ER_STAGGER}
	    	}\\
	    \subfloat[][ROC curves for the three \texttt{STAGGER} concepts.]{
	    		\includegraphics[width=1\textwidth]{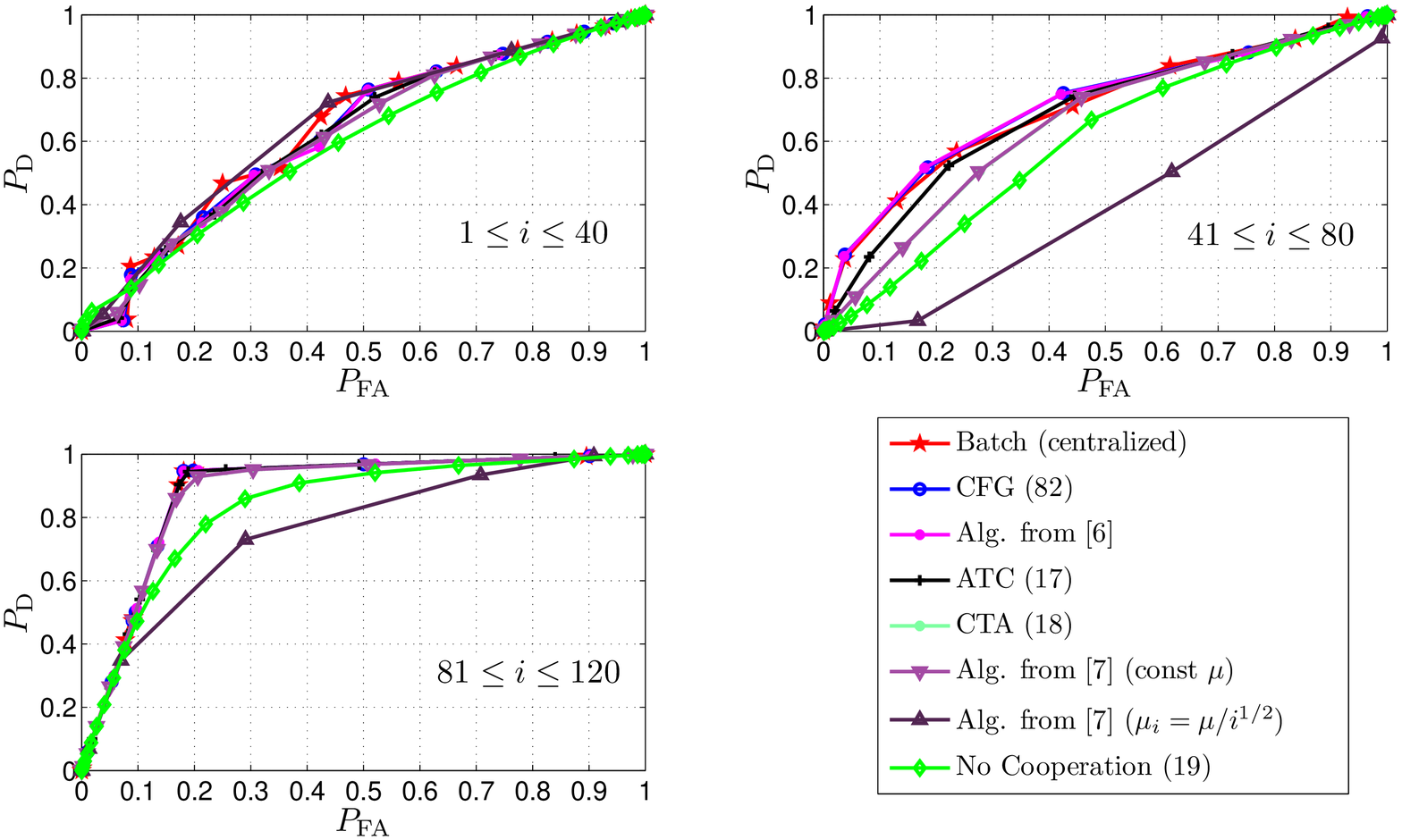}
	    		\label{fig:ROC_STAGGER}
	    }
	    \caption{Results from \texttt{STAGGER} simulation.}
	    \label{fig:STAGGER}
    \end{figure*}
\section{General Discussion}
\label{sec:general_discussion}
We saw in Sec.~\ref{sec:assumptions} that the excess-risk of a classifier can be written as a weighted mean-square-error with a weight matrix chosen according to Table \ref{tbl:metrics} when the step-size $\mu$ is small. This formulation of the excess-risk allows us to study the performance of distributed algorithms and explain their behavior. When the environment is stationary (for example, when the learners are sampling from a fixed distribution), we saw that the ATC and CTA diffusion algorithms can achieve an excess-risk performance proportional to $\mu$. In addition, we established that the ATC algorithm will outperform the CTA algorithm and non-cooperative processing when the combination matrix $A$ is doubly-stochastic. This generalizes previous results that only applied when the loss function used in the learning process is quadratic \cite{book_chapter}.  

When the environment is non-stationary, we modeled the optimizer $\w^o$ to be a random walk with i.i.d. increments. This model allows us to study the performance of the diffusion algorithm when tracking a non-stationary random process. We obtained (in Theorem \ref{thm:tracking}) a bound on the excess-risk that is comprised of three terms: a constant term that depends on the covariance matrix of the increments of the random walk process, a term that is proportional to $\mu$, and a term that is inversely proportional to $\mu$. This result is intuitive since we expect the diffusion algorithm to be able to track a fixed optimizer, or a relatively slow optimizer. As the optimizer evolves more quickly, however, the algorithm must increase the step-size in order to become more agile. The trade-off for the tracking ability of the diffusion algorithm is summarized in Fig.~\ref{fig:tracking_tradeoff}.

The simulation results illustrated that the steady-state excess-risk performance of the diffusion algorithm is proportional to the step-size $\mu$ (see Fig.~\ref{fig:ER_webspam_large_mu}-\ref{fig:ER_webspam_small_mu}). Furthermore, we showed through extensive simulations that the ATC algorithm outperforms the consensus-based algorithm proposed in \cite{yan} when constant step-sizes are employed. It can be observed from Fig.~\ref{fig:Stationary_ROC} that the area under the ROC curve of the ATC algorithm is larger than that of the non-cooperative, consensus-based, and CTA algorithms. Furthermore, the performance of the ATC algorithm is seen to approach that of batch processing, especially for small step-sizes. In Fig.~\ref{fig:ER_markov}, we see that a constant step-size algorithm can track a changing optimizer, unlike a diminishing step-size algorithm such as the one described in \cite{yan}. 

\section{Conclusion}
We analyzed the generalization ability of distributed online learning algorithms by showing that constant step-size algorithms can have bounded network excess-risk in non-stationary environments. We provided closed-form expressions for the asymptotic excess-risk and showed the advantage of cooperation over  networks.

\section{Acknowledgments}
\noindent Partial support for this project was received from the National Science Foundation grants CCF-1011918 and  CCF-0942936.

\appendix

\section{Comparing Diffusion and Non-Cooperative\\ Strategies}
\label{app:app4}
\subsection{CTA vs. Non-Cooperative Processing}
We confine our discussion to the following diffusion models
\begin{align}
	C=I_N,\quad A_1 = A,\quad A_2 = I_N\qquad\mathrm{(CTA)} \label{eq:CTA_cond1}\\
	C=I_N,\quad A_1 = I_N,\quad A_2 = A\qquad\mathrm{(ATC)} \label{eq:ATC_cond1}
\end{align}
The case of non-cooperating nodes corresponds to the choices:
\begin{align}
	C=I_N, \quad A_1 = I_N,\quad A_2 = I_N\quad\!\mathrm{({non-cooperative}\ processing)}
	\label{eq:ind_processing}
\end{align}
Our objective is to compare the network excess-risk achieved by the diffusion strategies and the excess-risk achieved when there is no cooperation between the nodes. We will conduct the analysis for constant step-sizes in stationary environments. To begin, we start from \eqref{eq:ER_SS_approx} and rewrite it as:
\begin{align}
	\E\|\tilde{\w}_{i-1}\|^2_\mathcal{T} \approx \mathrm{vec}(\mathcal{Y})^\T (I-\F)^{-1} \mathrm{vec}(\mathcal{T})
\end{align}
where
\begin{align}
	\mathcal{Y} \triangleq \mu^2 \R_v
\end{align}
We now perform the series expansion of $(I-\F)^{-1}$ to get
\begin{align}
	\E\|\tilde{\w}_{i-1}\|^2_\mathcal{T}  &\approx \mathrm{vec}(\mathcal{Y})^\T \sum_{j=0}^\infty \F^{j} \mathrm{vec}(\mathcal{T})\nonumber\\
							&= \mathrm{vec}(\mathcal{T})^\T \sum_{j=0}^\infty (\F^{j})^\T \mathrm{vec}(\mathcal{Y})\nonumber\\
							&= \mathrm{vec}(\mathcal{T})^\T \sum_{j=0}^\infty (\B^j\otimes \B^j) \mathrm{vec}(\mathcal{Y})\nonumber\\
							&= \sum_{j=0}^\infty \mathrm{vec}(\mathcal{T})^\T\mathrm{vec}(\B^j \mathcal{Y} (\B^j)^\T)\nonumber\\
							&= \sum_{j=0}^\infty \Tr(\mathcal{T}^\T \B^j \mathcal{Y} (\B^j)^\T) \label{eq:WMSE_general}
\end{align}
When Assumption \ref{ass:common_cost} holds, we have the weighting matrix $\mathcal{T}$ has the form $\mathcal{T} = I_N \otimes S$ where $S \triangleq \frac{1}{2N} \nabla^2 J(w^o)$. We can then simplify the above as:
\begin{align}
	\E\|\tilde{\w}_{i-1}\|^2_\mathcal{T}  &\approx \sum_{j=0}^\infty \Tr((I_N \otimes S) \B^j \mathcal{Y} (\B^j)^\T)
\end{align}
In addition, with Assumption \ref{ass:common_cost}, we have $\mathcal{D} = I_N \otimes D^o$ for some $M\times M$ matrix $D^o$ that is the same for all nodes, then we can further write:
\begin{align}
	\B = A_1^\T \otimes (I_M - \mu D^o)
\end{align}
We define the excess-risk for CTA and non-cooperative processing as:
\begin{align}
	\mathrm{ER}_{\mathrm{ind}} &\triangleq \mu^2 \sum_{j=0}^\infty \Tr((I_N \otimes S) \B_{\mathrm{ind}}^j \mathcal{Y} \B_{\mathrm{ind}}^{\T j})\\
	\mathrm{ER}_{\mathrm{CTA}} &\triangleq \mu^2 \sum_{j=0}^\infty \Tr((I_N \otimes S) \B_{\mathrm{CTA}}^j \mathcal{Y} \B_{\mathrm{CTA}}^{\T j})
\end{align}
where $\B_{\mathrm{CTA}}$ and $\B_{\mathrm{ind}}$ are defined as:
\begin{align}
	\B_{\mathrm{ind}} \triangleq I_N \otimes (I_M - \mu D^o) \label{eq:B_ind}\\
	\B_{\mathrm{CTA}} \triangleq A \otimes (I_M - \mu D^o) \label{eq:B_CTA}
\end{align}
Noticing that $\mathcal{Y}$ is the same for CTA and the individual processing case, we compute the difference in the excess-risk as:
\begin{align}
	&\mathrm{ER}_{\mathrm{ind}} - \mathrm{ER}_{\mathrm{CTA}} = \nonumber\\
&\mu^2 \sum_{j=0}^\infty \Tr((\B_{\mathrm{ind}} (I_N \otimes S) \B_{\mathrm{ind}}^\T - \B_{\mathrm{CTA}} (I_N \otimes S) \B_{\mathrm{CTA}}^\T)\mathcal{Y})
	\label{eq:WMSE1}
\end{align}
We substitute \eqref{eq:B_ind}-\eqref{eq:B_CTA} into \eqref{eq:WMSE1}, and get:
\begin{align}
	&\mathrm{ER}_{\mathrm{ind}} - \mathrm{ER}_{\mathrm{CTA}} = \nonumber\\
	&\mu^2 \sum_{j=0}^\infty \Tr(((I_N - A^jA^{j\T})\otimes ((I_M - \mu D^o)^j S (I_M - \mu D^o)^j))\mathcal{Y})
\end{align}
Since $S \triangleq \frac{1}{2N} \nabla^2 J(w^o)$ is positive-definite, we conclude that $(I_M - \mu D^o)^j S (I_M - \mu D^o)^j \geq 0$. Finally, since we assumed that $A$ is doubly-stochastic, then $A^j$ is also doubly-stochastic, as well as $A^j A^{j\T}$. Therefore, the matrix $(I-A^j A^{j\T}) \geq 0$ and its eigenvalues are in the range $[0,1]$ \cite{book_chapter}. Finally, combining these facts with the knowledge that $\mathcal{Y} \geq 0$, we conclude that:
\begin{align}
\boxed{
	\mathrm{ER}_{\mathrm{CTA}} \leq \mathrm{ER}_{\mathrm{ind}}
	}
\ \text{(small $\mu$, large $i$, $C=I_N$, $\mathds{1}^\T A = \mathds{1}^T$, $A \mathds{1} = \mathds{1}$)}
\end{align}
A similar conclusion holds for ATC. Actually, ATC outperforms CTA as well, as we show next. 

\subsection{ATC vs. CTA}
In order to compare ATC to CTA, we continue our assumption that the matrix $A$ is doubly stochastic, but we generalize our model for CTA and ATC from \eqref{eq:CTA_cond1} and \eqref{eq:ATC_cond1} to:
\begin{align}
	C,\quad A_1 = A,\quad A_2 = I_N\qquad\mathrm{(CTA)} \label{eq:CTA_cond2}\\
	C,\quad A_1 = I_N,\quad A_2 = A\qquad\mathrm{(ATC)} \label{eq:ATC_cond2}
\end{align}
where we have modified the model to allow for an arbitrary right-stochastic matrix $C$. We continue from \eqref{eq:WMSE_general} and rewrite the network excess-risk at steady-state for both CTA and ATC as:
\begin{align}
	\mathrm{ER}_{\mathrm{CTA}} &= \sum_{j=0}^\infty \Tr(\mathcal{T}^\T \B_{\mathrm{CTA}}^j \mathcal{Y}_{\mathrm{CTA}} (\B_{\mathrm{CTA}}^j)^\T) \label{eq:WMSE_CTA1}\\
	\mathrm{ER}_{\mathrm{ATC}} &= \sum_{j=0}^\infty \Tr(\mathcal{T}^\T \B_{\mathrm{ATC}}^j \mathcal{Y}_{\mathrm{ATC}} (\B_{\mathrm{ATC}}^j)^\T) \label{eq:WMSE_ATC1}
\end{align}
where
\begin{align}
	\B_{\mathrm{CTA}} &\triangleq [ I_{M N} - \mu \mathcal{D} ] \mathcal{A}^\T \label{eq:B_CTA2}\\
	\B_{\mathrm{ATC}} &\triangleq \mathcal{A}^\T [ I_{M N} - \mu \mathcal{D} ] \label{eq:B_ATC2}\\
	\mathcal{Y}_{\mathrm{CTA}} &\triangleq \mu^2 \R_v\\
	\mathcal{Y}_{\mathrm{ATC}} &\triangleq \mu^2 \mathcal{A}^\T \R_v \mathcal{A}
\end{align}
Like the previous section, we assume the same risk function for all nodes (i.e., Assumption \ref{ass:common_cost} holds) so that $\mathcal{D} = I_N \otimes D^o$ and that the weighting matrix $\mathcal{T}$ has the form $\mathcal{T}=I_N \otimes S$ where $S \triangleq \frac{1}{2N} \nabla^2 J(w^o)$. With the first assumption, we have:
\begin{align}
	\B_{\mathrm{CTA}} = \B_{\mathrm{ATC}} = A^T \otimes (I_M - \mu D^o)
\end{align}
We compute the difference between the excess-risks:
\begin{align*}
	&\mathrm{ER}_{\mathrm{CTA}} \!\!-\!\! \mathrm{ER}_{\mathrm{ATC}} \nonumber\\
	&= \sum_{j=0}^\infty \!\Tr\left( \left(A^j (I\!-\!AA^{\T}) A^{j\T}\right)\!\! \otimes\!\! \left((I_M\!-\!\mu D^o)^j S (I_M \!-\! \mu D^o)^j\right)\!\mu^2 \R_v\right)
\end{align*}
We can verify that the above difference is non-negative by noting that $\R_v > 0$ and $(I_M-\mu D^o)^j S (I_M - \mu D^o)^j$ is positive-semi-definite. Moreover, $A^j (I-AA^{\T}) A^{j\T} \geq 0$ \cite{book_chapter}. Therefore, we have established, under our assumptions, that
\begin{align}
	\mathrm{ER}_{\mathrm{ATC}} \leq \mathrm{ER}_{\mathrm{CTA}}
\end{align}
Therefore, combining this result with the result from the previous appendix we conclude that for small $\mu$, large $i$, $C=I_N$, $\mathds{1}^\T A = \mathds{1}^\T$, and $A \mathds{1} = \mathds{1}$
\begin{align}
\boxed{
	\mathrm{ER}_{\mathrm{ATC}} \leq \mathrm{ER}_{\mathrm{CTA}} \leq \mathrm{ER}_{\mathrm{ind}}
}
\end{align}

\section{Mean-Square-Error Analysis}
\label{app:proof_convergence}
We follow the approach of \cite{Jianshu_common_wo} and extend it to handle non-stationary environments as well. We define the error vectors at node $k$ at time $i$ as:
\begin{align}
	\tilde{\bphi}_{k,i} &\triangleq \w^o_{i} - \bphi_{k,i} \label{eq:bphi_ki}\\
	\tilde{\bpsi}_{k,i} &\triangleq \w^o_i - \bpsi_{k,i}\\
	\tilde{\w}_{k,i}^f    &\triangleq \w^o_i -\w_{k,i} \label{eq:bw_ki}
\end{align}
We subtract \eqref{eq:C1} from $\w^o_{i-1}$ and \eqref{eq:A}-\eqref{eq:C2} from $\w^o_i$ using \eqref{eq:grad_model} to get
\begin{align}
	\tilde{\bphi}_{k,i-1} &= \sum_{\ell=1}^N a_{1,\ell k} \tilde{\w}_{\ell,i-1}^f \label{eq:C1_err}\\
	\tilde{\bpsi}_{k,i} &= \tilde{\bphi}_{k,i-1} + \q_i + \mu \sum_{\ell = 1}^N c_{\ell k} \left[\nabla J_{\ell,i-1}(\bphi_{k,i-1}) + \v_\ell(\bphi_{k,i-1})\right] \label{eq:A_err}\\
	\tilde{\w}_{k,i}^f   &= \sum_{\ell=1}^N a_{2,\ell k} \tilde{\bpsi}_{\ell,i} \label{eq:C2_err}
\end{align}
Using the mean-value-theorem for real vectors \eqref{eq:polyak1}, we can express the gradient $\nabla J_{k,i-1}(\bphi_{k,i-1})$ in terms of $\tilde{\bphi}_{k,i-1}$:
\begin{align}
	\nabla J_{\ell,i-1}(\bphi_{k,i-1})\! &=\!\!\nabla \!J_{\ell,i-1}(\w^o_{i-1})\! -\! \left[\int_0^1 \!\!\!\!\!\nabla^2 J_{\ell,i-1}(\w^o_{i-1}\! -\! t \tilde{\bphi}_{k,i-1})dt\right]   \tilde{\bphi}_{k,i-1} \nonumber\\
								 &\triangleq -\H_{\ell,k,i}\tilde{\bphi}_{k,i-1}
	\label{eq:gradient_expansion}
\end{align}
where we are defining
\begin{align}
	\H_{\ell,k,i} \triangleq \int_0^1 \nabla^2 J_{\ell,i-1}(\w^o_{i-1}-t \tilde{\bphi}_{k,i-1})dt
\end{align}
Notice that $\nabla J_{\ell,i-1} (\w^o_{i-1}) = 0$ since the minimizer at time $i-1$ is $\w^o_{i-1}$. Substituting \eqref{eq:gradient_expansion} into \eqref{eq:A_err}, we get
\begin{align}
	\tilde{\bpsi}_{k,i} &= \left[I-\mu \sum_{\ell = 1}^N c_{\ell k} \H_{\ell,k,i-1}\right] \tilde{\bphi}_{k,i-1} + \mu \sum_{\ell = 1}^N c_{\ell k} \v_\ell(\bphi_{k,i-1}) + \q_{i}
	\label{eq:A2_err}
\end{align}

\subsection{Local MSE Recursions}
We now derive the mean-square-error (MSE) recursions by noting that the squared norm $\|x\|^2 \triangleq x^\T x$ is a convex function of $x$. Therefore, applying Jensen's inequality \cite[p.77]{cvx_book} to \eqref{eq:bphi_ki} and \eqref{eq:bw_ki}  we get:
\begin{align}
	\mathbb{E} \|\tilde{\bphi}_{k,i-1}\|^2 &\leq \sum_{\ell=1}^N a_{1,\ell k} \mathbb{E}\|\tilde{\w}_{\ell,i-1}^f\|^2,\quad k=1,\ldots,N \label{eq:C1_var_ineq}\\
	\mathbb{E} \|\tilde{\w}_{k,i}^f\|^2 &\leq \sum_{\ell=1}^N a_{2,\ell k} \mathbb{E}\|\tilde{\bpsi}_{\ell,i}\|^2,\quad\quad k=1,\ldots,N \label{eq:C2_var_ineq}
\end{align}
From \eqref{eq:A2_err} and using Assumption \ref{ass:noiseModeling}, we obtain
\begin{align}
	\mathbb{E} \|\tilde{\bpsi}_{k,i}\|^2 &= \mathbb{E} \|\tilde{\bphi}_{k,i-1}\|^2_{\bSigma_{k,i}}  + \mathbb{E} \|\q_{i}\|^2 + \mu^2 \E\left\Vert\sum_{\ell = 1}^N c_{\ell k} \v_\ell(\bphi_{k,i-1})\right\Vert^2
	\label{eq:A_var_eq}
\end{align}
where we are introducing the weighting matrix:
\begin{align}
	\bSigma_{k,i} &\triangleq \left(I_M - \mu \sum_{\ell = 1}^N c_{\ell k} \H_{\ell,k,i}\right)^2
\end{align}
The matrices $\bSigma_{k,i}$ are positive semi-definite and bounded by:
\begin{align}
	0 \leq \bSigma_{k,i} \leq \gamma_{k}^2 I_M
	\label{eq:bSigma_bound}
\end{align}
where
\begin{align}
	\gamma_{k} \triangleq \max\left\{\left|1-\mu \lambda_{\max}\sum_{\ell = 1}^N c_{\ell k}\right|,\left|1-\mu \lambda_{\min}\sum_{\ell = 1}^N c_{\ell k}\right|\right\}
	\label{eq:gamma}
\end{align}
Now note that the square of $\gamma_{k}$ from \eqref{eq:gamma} can be upper-bounded by:
\begin{align}
	\gamma_{k}^2 &= \max\left\{1-2 \mu \lambda_{\max} \sum_{\ell = 1}^N c_{\ell k} + \mu^2 \lambda_{\max}^2 \left( \sum_{\ell = 1}^N c_{\ell k}\right)^2,\right. \nonumber\\
				   &\quad\quad\quad\quad\left.1- 2\mu \lambda_{\min} \sum_{\ell = 1}^N c_{\ell k} + \mu^2 \lambda_{\min}^2 \left(\sum_{\ell = 1}^N c_{\ell k}\right)^2\right\} \nonumber\\
				   &\leq 1-2 \mu \lambda_{\min} C_* + \mu^2 \lambda_{\max}^2 \|C\|_1^2
\end{align}
where $C_*$ denotes the minimum absolute column sum of the matrix $C$. In order to simplify the notation in the following analysis, we introduce the upper-bound
\begin{align}
	\beta \triangleq 1-2 \mu \lambda_{\min} C_* + \mu^2 \lambda'
\end{align}
where
\begin{align}
	\lambda' \triangleq (\lambda_{\max}^2 + \alpha) \|C\|_1^2
\end{align}
and $\alpha$ is defined in Assumption \ref{ass:noiseModeling}. Also, note that by Lemma 3 from \cite{Jianshu_common_wo}, we have:
\begin{align}
	\E\left\Vert\sum_{\ell = 1}^N c_{\ell k} \v_\ell(\bphi_{k,i-1})\right\Vert^2 \leq \|C\|_1^2 \left[\alpha \E\|\tilde{\bphi}_{k,i-1}\|^2+\sigma_v^2\right]
	\label{eq:noise_var_bound}
\end{align}
Combining \eqref{eq:bSigma_bound}, \eqref{eq:noise_var_bound}, and \eqref{eq:A_var_eq}, we obtain for all $k=1,\ldots,N$:
\begin{align}
	\mathbb{E} \|\tilde{\bpsi}_{k,i}\|^2 &\leq \beta\, \E\|\tilde{\bphi}_{k,i-1}\|^2 + \mu^2 \|C\|_1^2\sigma_v^2 + \Tr(Q)
	\label{eq:A_var_ineq}
\end{align}

\subsection{Network MSE Recursions}
We now combine the MSE values at each node into network MSE vectors as follows:
\begin{align}
	\mathcal{W}_i &\triangleq \big[\E\|\tilde{\w}_{1,i}^f\|^2,\dots,\E\|\tilde{\w}_{N,i}^f\|^2\big]^\T\\
	\mathcal{X}_i &\triangleq \big[\E\|\tilde{\bphi}_{1,i}\|^2,\dots,\E\|\tilde{\bphi}_{N,i}\|^2\big]^\T\\
	\mathcal{Y}_i &\triangleq \big[\E\|\tilde{\bpsi}_{1,i}\|^2,\dots,\E\|\tilde{\bpsi}_{N,i}\|^2\big]^\T
\end{align}
We can then rewrite \eqref{eq:C1_var_ineq}, \eqref{eq:A_var_ineq}, and \eqref{eq:C2_var_ineq} as:
\begin{align}
	\mathcal{X}_{i-1} &\preceq A_1^\T \mathcal{W}_{i-1}\\
	\mathcal{Y}_i &\preceq \beta \mathcal{X}_{i-1} + (\mu^2 \|C\|_1^2 \sigma_v^2 + \Tr(Q))\mathds{1}_N\\
	\mathcal{W}_i &\preceq A_2^\T \mathcal{Y}_{i}
\end{align}
where $x \preceq y$ indicates that each element of the vector $x$ is less than or equal to the correspondent element of vector $y$. Moreover, the notation $\mathds{1}_N$ denotes the vector with all entries equal to one. Using the fact that if $x \preceq y$ then $B x \preceq B y$ for any matrix $B$ with non-negative entries, we can combine the above inequality recursions into a single recursion for $\mathcal{W}_i$ and get:
\begin{align}
	\mathcal{W}_i &\preceq \beta A_2^\T A_1^\T \mathcal{W}_{i-1} + ( \mu^2 \|C\|_1^2 \sigma_v^2 + \Tr(Q)) \mathds{1}_N
				  \label{eq:bigW_recur_1}
\end{align}
We now upper-bound the $\infty$-norm (maximum absolute value) of the vector $\mathcal{W}_i$ in order to obtain the scalar-recursion:
\begin{align*}
	\|\mathcal{W}_i\|_\infty &\leq \|\beta A_2^\T A_1^\T \mathcal{W}_{i-1}\|_\infty + \mu^2 \|C\|_1^2 \sigma_v^2 +  \Tr(Q)\nonumber\\
							 &\leq \beta \cdot \|A_2^\T\|_\infty \cdot \|A_1^\T\|_\infty \cdot \|\mathcal{W}_{i-1}\|_\infty + \mu^2 \|C\|_1^2 \sigma_v^2 +  \Tr(Q)
\end{align*}
where $\|A\|_\infty$ denotes the maximum absolute row sum of matrix $A$. Noting that the matrices $A_1$ and $A_2$ are left-stochastic, we have that $\|A_1^\T\|_\infty = 1$ and $\|A_1^\T\|_\infty = 1$. Therefore,
\begin{align}
	\|\mathcal{W}_i\|_\infty &\leq \beta \|\mathcal{W}_{i-1}\|_\infty + \|C\|_1^2 \sigma_v^2 \mu^2 + \Tr(Q)
	\label{eq:scalar_recursion}
\end{align}
Unrolling \eqref{eq:scalar_recursion}, we get
\begin{equation}
\boxed{
	\|\mathcal{W}_i\|_\infty \leq \beta^i \|\mathcal{W}_{0}\|_\infty +
							 \left(\|C\|_1^2 \sigma_v^2 \mu^2 + \Tr(Q)\right) \sum_{j=0}^{i-1} \beta^j
	}
	\label{eq:intermediate_rec_ineq}
\end{equation}

\bibliographystyle{elsarticle-num}
\bibliography{refs_Neurocomputing}

%\end{footnotesize}

% ****************************************************************************
% END OF BIBLIOGRAPHY AREA
% ****************************************************************************

\end{document}